\documentclass[10pt,reqno]{amsart}
\usepackage[colorlinks=true,
pdfstartview=FitV, linkcolor=cyan, citecolor=magenta,
urlcolor=]{hyperref}
\usepackage[dvipsnames]{xcolor}
\usepackage{amsmath,amsfonts,latexsym,amssymb}
\usepackage{amssymb,amsfonts, amsmath}
\usepackage{mathrsfs}
\usepackage[latin1]{inputenc}
\usepackage[T1]{fontenc}
\usepackage{comment}
 \usepackage{mathtools} 
\usepackage[all,arc]{xy}
\usepackage{enumerate}
\usepackage{mathrsfs}
\usepackage{mathabx}
\usepackage{amssymb}
\usepackage{graphicx}
\usepackage{subfig}
\usepackage{verbatim} 
\usepackage[all,arc]{xy}
\usepackage{enumerate}
\usepackage{mathrsfs}
\usepackage[bb = fourier,cal = euler,scr = rsfs]{mathalfa}	
\usepackage{hyperref}
\usepackage{amsmath}
\usepackage[makeroom]{cancel}
\usepackage{todonotes}
\makeatletter
\newsavebox{\@brx}
\newcommand{\llangle}[1][]{\savebox{\@brx}{\(\m@th{#1\langle}\)}%
  \mathopen{\copy\@brx\mkern2mu\kern-1\wd\@brx\usebox{\@brx}}}
\newcommand{\rrangle}[1][]{\savebox{\@brx}{\(\m@th{#1\rangle}\)}%
  \mathclose{\copy\@brx\mkern2mu\kern-1\wd\@brx\usebox{\@brx}}}
\makeatother
\newtheorem{theorem}{Theorem}[section]
\newtheorem{lemma}[theorem]{Lemma}
\newtheorem{proposition}[theorem]{Proposition}
\newtheorem{corollary}[theorem]{Corollary}

\newtheorem{remark}[theorem]{\normalfont{\em{Remark}}}
\newtheorem{Exmp}[theorem]{\normalfont{\em{Example}}}

\renewcommand{\leq}{\leqslant}
\renewcommand{\geq}{\geqslant}

\usepackage{braket}
\usepackage{amsmath}
\usepackage[title]{appendix}
\makeatletter
\renewcommand*\env@matrix[1][\arraystretch]{%
\edef\arraystretch{#1}%
\hskip -\arraycolsep
\let\@ifnextchar\new@ifnextchar
\array{*\c@MaxMatrixCols c}}
\makeatother
\title{Matrix entries, unipotents, and linearity of amalgams}
\date{\today}

\author{Sami Douba}
\address{Mathematisches Institut der Universit\"at Bonn, Endenicher Allee 60, 53115 Bonn, Germany}
\email{douba@math.uni-bonn.de}

\author{Konstantinos Tsouvalas}
\address{Max Planck Institute for Mathematics in the Sciences, Inselstrasse 22, 04103 Leipzig, Germany}
\email{konstantinos.tsouvalas@mis.mpg.de}

\begin{document}

\frenchspacing

\maketitle

\begin{abstract} 
We investigate linearity of amalgams of subgroups of algebraic groups along intersections with algebraic subgroups. In the process, we establish linearity of certain ``doubles'' of linear groups, and obtain new examples of finitely generated residually finite groups that fail to be linear.
\end{abstract}

\section{Introduction}
An abstract group $\Gamma$ is {\em linear over} an integral domain $\mathsf{R}$ (or is {\em $\mathsf{R}$-linear})  if $\Gamma$ embeds in $\mathsf{GL}_n(\mathsf{R})$ for some $n \in \mathbb{N}$. One says $\Gamma$ is {\em linear} if $\Gamma$ is linear over some integral domain. The goal of the current paper is to isolate contexts in which certain groups constructed from linear groups remain, or cease to be, linear.

It has been known since the 40's that a free product of two linear groups is linear \cite{Nisnevich, Shalen}. Determining the smallest integral domain over which a free product of two linear groups remains linear is however a much subtler problem. For example, it was only recently established that a free product of two $\mathbb{Z}$-linear groups is $\mathbb{Z}$-linear \cite[Cor.~1.10]{danciger2024combination}.

On the other hand, linearity of an amalgam of two linear groups may fail dramatically. For example, if~$\Lambda$ is a subgroup of a residually finite group $\Gamma$, then the {\em double $\Gamma \ast_\Lambda \Gamma$ of $\Gamma$ along~$\Lambda$}, that is, the amalgam given by the inclusion into either factor, is residually finite if and only if~$\Lambda$ is separable in $\Gamma$. (Recall that $\Lambda$ is said to be {\em separable} in $\Gamma$ if $\Lambda$ is an intersection of finite-index subgroups of $\Gamma$, and that $\Gamma$ is said to be {\em residually finite} if the trivial subgroup of $\Gamma$ is separable.) On the other hand, finitely generated linear groups are residually finite by Mal'cev's theorem~\cite{Malcev}. It is then not difficult to deduce from the Tits alternative~\cite{MR286898}, together with the abundance of finitely generated groups that are not residually finite \cite{MR2029029}, that any finitely generated linear group that is not virtually solvable possesses uncountably many subgroups the doubles along which fail to be residually finite, let alone linear.

If $\mathsf{R}$ is a finitely generated integral domain and $\mathsf{G}$ is an algebraic subgroup of $\mathsf{GL}_n$, then $\mathsf{G}(\mathsf{R})$ is separable in $\mathsf{GL}_n(\mathsf{R})$; see \cite{MR1769939, MR898729}. In particular, if $\mathsf{R}$ is an arbitrary integral domain and $\Gamma < \mathsf{GL}_n(\mathsf{R})$ is finitely generated, then any intersection $\Lambda$ of $\Gamma$ with an algebraic subgroup of $\mathsf{GL}_n(\mathsf{R})$ is separable in $\Gamma$. One may then ask whether there are examples of such pairs $(\Gamma, \Lambda)$ for which the double $\Gamma \ast_{\Lambda} \Gamma$ nevertheless fails to be linear. Such examples were supplied by Wehrfritz \cite[Cor.~2.4]{Wehrfritz} and rediscovered by Dru\textcommabelow{t}u--Sapir~\cite{MR2115010}. In the present work, we use the superrigidity theorems of Margulis \cite{Mar}, Corlette \cite{Corlette}, and Gromov--Schoen \cite{GS} to provide new examples via the following theorem.

\begin{theorem}\label{negative}
Let $\mathsf{G}$ be a semisimple real algebraic group with no compact factors that is not locally isomorphic to $\mathsf{O}(n,1)$ or $\mathsf{U}(n,1)$ for any $n\in \mathbb{N}$, and let $\Gamma < \mathsf{G}$ be an irreducible lattice. If~$\Lambda$ is an infinite-index subgroup of $\Gamma$ that is not cocompact in the Zariski-closure of $\Lambda$ in $\mathsf{G}$, then the double $\Gamma \ast_\Lambda \Gamma$ is not linear.
\end{theorem}

Thus, for example, the double $\mathsf{SL}_n(\mathbb{Z}) \ast_{\mathsf{SL}_m(\mathbb{Z})} \mathsf{SL}_n(\mathbb{Z})$ is residually finite but fails to be linear for $2 \leq m \leq n-1$ (where $\mathsf{SL}_m(\mathbb{Z})$ is embedded in $\mathsf{SL}_n(\mathbb{Z})$, for instance, by adding $1$'s on the diagonal), as does the residually finite double $\mathsf{SL}_2(\mathbb{Z}[\sqrt{2}]) \ast_{\mathsf{SL}_2(\mathbb{Z})} \mathsf{SL}_2(\mathbb{Z}[\sqrt{2}])$. We refer to~\cite{MR764305, MR4388367, MR4653708, MR4862320} for other examples of nonlinear finitely generated residually finite groups.

On the other hand, the following positive result establishes linearity of a large class of doubles of the form $\Gamma \ast_\Lambda \Gamma$, where $\Gamma$ is arithmetic and $\Lambda < \Gamma$ is \hbox{cocompact in the Zariski-closure of~$\Lambda$.}

\begin{theorem}\label{compact-1}  Let $n\geq 2$ be an integer and $\mathsf{C}<\mathsf{SL}_n(\mathbb{R})$ a compact subgroup. 
Let $\mathsf{C}$ be defined by a set of polynomials of degree at most $r$ in the matrix entries and with coefficients in a subfield $K \subset \mathbb{R}$, and let $\Lambda := \mathsf{SL}_n(K)\cap \mathsf{C}$. Then there is a faithful representation of the double ${\mathsf{SL}_n(K) \ast_\Lambda \mathsf{SL}_n(K})$ into $\mathsf{SL}_m(K[s,t])$, where $K[s,t]$ denotes the ring of polynomials in two variables over~$K$ and ${m=2^{4\binom{n^2+r}{r}}}$.
\end{theorem}

The following corollary is a consequence.

 \begin{corollary}\label{compactcorollaries} \begin{enumerate}\item\label{compact-2} Let $V$ be a finite-dimensional real vector space, let $\Gamma < \mathsf{GL}(V)$, and let $\mathsf{L} < \mathsf{GL}(V)$ be compact. Then for every finite-index subgroup $\Lambda$ of $\Gamma \cap \mathsf{L}$ that is separable in $\Gamma$, the double $\Gamma \ast_{\Lambda}\Gamma$ is  linear over $\mathbb{R}$. 
\item \label{precompactgaloisconjugate}
Let $\Gamma < \mathsf{GL}_n(\mathbb{C})$ be a subgroup. Let $\mathsf{G}$ be a $\mathbb{C}$-algebraic subgroup of $\mathsf{GL}_n(\mathbb{C})$, and suppose~$\Lambda$ is a finite-index subgroup of $\Gamma \cap \mathsf{G}$ that is separable in $\Gamma$. If there is an embedding $\sigma: K \rightarrow \mathbb{C}$ such that $(\Gamma\cap \mathsf{G})^\sigma$ is precompact in $\mathsf{GL}_n(\mathbb{C})$, where $K$ denotes the entry field of $\Gamma$, then the double $\Gamma \ast_\Lambda \Gamma$ is linear over $\mathbb{R}$.\end{enumerate}
\end{corollary}

It thus follows, for example, that if $\Gamma$ is a finitely generated subgroup of a compact Lie group, then the double $\Gamma \ast_{\Lambda}\Gamma$ is linear for any maximal abelian subgroup $\Lambda$ of $\Gamma$. Note also that, in the setting of Corollary~\ref{compactcorollaries}, separability of $\Lambda$ in $\Gamma$ is not automatic even for finitely generated~$\Gamma$. Indeed, if $m \geq 5$ and one takes $\Gamma:= \mathsf{Spin}(f_m; \mathbb{Z}[\sqrt{2}])$, where $f_m$ is the quadratic form $\sqrt{2}x_1^2+\sqrt{2}x_2^2+x_3^2+\ldots+x_m^2$, and $\Lambda$ to be the stabilizer in $\Gamma$ of the standard basis vector~$e_1$, then it follows from work of Kneser~\cite{MR549966} and Millson~\cite{MR422501} related to the congruence subgroup property that $\Lambda$ possesses finite-index subgroups that fail to be separable in $\Gamma$. This example also demonstrates the importance of the separability assumptions in Theorems~\ref{absolutelysimple} and~\ref{anisotropic} to follow.

While the stipulation that a finitely generated group embed in a compact Lie group may appear very restrictive, the authors are for instance not aware of any linear (Gromov-)hyperbolic group that does not embed in a compact Lie group. Indeed, celebrated work of Agol \cite{MR3104553} established that any hyperbolic group that acts properly and cocompactly on a $\mathrm{CAT}(0)$ cube complex is virtually compact special. Following Haglund--Wise \cite{HW08}, a finitely generated group~$\Gamma$ is said to be {\em special} (respectively, {\em compact special}) if $\Gamma$ embeds in a right-angled Coxeter group (resp., embeds as a quasiconvex subgroup of a right-angled Coxeter group $W$ with respect to a Coxeter basis for $W$); see \cite{HW08} for several equivalent definitions. Agol \cite{agol2018hyperbolic} showed that any finitely generated right-angled Coxeter group, and hence any finitely generated virtually special group, embeds in a compact Lie group. In \S \ref{amalgamsalongprecompact}, we apply the latter fact together with Corollary~\ref{compactcorollaries} to prove the following.

\begin{theorem}\label{specialoverQC}
Let $\Gamma$ be a virtually compact special Gromov-hyperbolic group and $\Lambda < \Gamma$ a quasiconvex subgroup. Then the double $\Gamma \ast_\Lambda \Gamma$ is linear over $\mathbb{R}$.
\end{theorem}

In the sequel, a generalization of Theorem~\ref{specialoverQC} (Theorem~\ref{virtuallyspecialrelativelyhyperbolic}) will be used, together with much existing knowledge, to show that the double of a finitely generated Kleinian group along an arbitrary finitely generated subgroup is linear (Corollary~\ref{kleinian}). Note that in the statement of Theorem~\ref{specialoverQC} there is no malnormality assumption on $\Lambda$, so that the double $\Gamma \ast_\Lambda \Gamma$ need not be hyperbolic. In a previous version of this article, we remarked that we nevertheless expected a double $\Gamma \ast_\Lambda \Gamma$ as in Theorem \ref{specialoverQC} to remain virtually compact special (compare, for instance,~\cite{MR2979855}), giving an alternative proof of linearity of such doubles. Since then, Changqian Li has communicated to us a proof, relying on work of Huang--Wise~\cite{MR4693936}, that such doubles are indeed virtually compact special, and also made public a preprint \cite{arXiv:2605.21734} where he proves the latter directly by showing that appropriately defined doubles of virtually special compact cube complexes (whose fundamental groups are not necessarily hyperbolic) are themselves virtually special.

Given a subgroup $\Lambda$ of a (linear) real algebraic group $\mathsf{G}$, we will denote by $\overline{\Lambda}^{\mathrm{Zar}}$ the Zariski-closure of $\Lambda$ in $\mathsf{G}$. As another application of Corollary~\ref{compactcorollaries}, we establish the following.

\begin{theorem}\label{absolutelysimple}
Let $\mathsf{G}$ be an absolutely almost simple real algebraic group, let $\Gamma < \mathsf{G}$ be a  lattice in $\mathsf{G}$ with adjoint trace field $\neq \mathbb{Q}$, and let $\Lambda < \Gamma$ be a subgroup. Then

\begin{enumerate}
\item\label{absolutelysimple1}  if $\Gamma$ is arithmetic and $\Lambda$ is cocompact in $\overline{\Lambda}^{\mathrm{Zar}}$ and separable in $\Gamma$, then the double $\Gamma \ast_{\Lambda} \Gamma$ is linear; 

\item\label{absolutelysimple2} if $\mathsf{G}$ is not locally isomorphic to $\mathsf{O}(n,1)$ or $\mathsf{U}(n,1)$ for any $n \in \mathbb{N}$, then the double $\Gamma \ast_\Lambda \Gamma$ is linear if and only if $\Lambda$ is cocompact in $\overline{\Lambda}^{\mathrm{Zar}}$ and separable in $\Gamma$.
\end{enumerate}
\end{theorem}

We remark for example that if $p \geq q \geq 2$ with $p+q \geq 5$ odd, then the adjoint trace field of any cocompact lattice in $\mathsf{O}(p,q)$ differs from $\mathbb{Q}$; see Remark \ref{indefiniteorthogonal}.

Corollary~\ref{compactcorollaries} also has applications to doubles of certain $S$-arithmetic groups, since the latter groups often embed in compact Lie groups. Throughout this article, given a number field $K$ and a place $v$ of $K$, we denote by $K_v$ the completion of $K$ with respect to $v$, and given a finite set $S$ of places of $K$, we denote by $\mathcal{O}_{K,S}$ the ring of $S$-integers of $K$. 

\begin{theorem}\label{anisotropic}
Let ${\bf G}$ be an almost simple algebraic group over a number field $K$, and $S$ be a finite set of places of $K$ containing all archimedean places $v$ of $K$ such that ${\bf G}$ is isotropic over $K_v$. Suppose ${\bf H}$ is a reductive $K$-subgroup of ${\bf G}$ that is anisotropic over $K_w$ for some archimedean place $w$ of $K$. Set $\mathsf{G} = \prod_{v \in S} {\bf G}(K_v)$ and $\mathsf{H} = \prod_{v \in S}{\bf H}(K_v)$.
Then for any lattice $\Gamma < \mathsf{G}$ commensurable with ${\bf G}(\mathcal{O}_{K,S})$ and any finite-index subgroup $\Lambda$ of $\Gamma \cap \mathsf{H}$ that is separable in $\Gamma$, the double $\Gamma \ast_\Lambda \Gamma$ is linear over $\mathbb{R}$. 
\end{theorem}

It is not clear to us whether in Theorem \ref{anisotropic} one can replace the condition that ${\bf H}$ be anisotropic over $K_w$ for some archimedean place $v_w$ of $K$, with the weaker condition that~${\bf H}$ be anisotropic over $K$ itself. We remark that if $\Gamma$ is an irreducible lattice in a product $\mathsf{G}$ of almost simple groups over nonarchimedean local fields of characteristic $0$, and if the sum of the ranks of the factors of $\mathsf{G}$ is $\geq 2$, then, by Margulis's arithmeticity theorem \cite{Mar}, we have that $\mathsf{G} = \prod_{v \in S} {\bf G}(K_v)$ and that~$\Gamma$ is commensurable with ${\bf G}(\mathcal{O}_{K,S})$ for some ${\bf G}$, $K$, and $S$ as in Theorem~\ref{anisotropic}, where $S$ consists entirely of nonarchimedean places of $K$, and in this case, one knows that~${\bf G}$ is $K_w$-anisotropic for any archimedean place $w$ of $K$, so that the latter also holds for any $K$-subgroup~${\bf H}$ of~${\bf G}$.

Our methods also allow us to establish linearity of doubles of certain matrix groups along unipotent-free subgroups in the following geometric setting. A submanifold $Y$ of an Hadamard manifold $X$ is said to be {\em reflective} if $Y$ is the fixed point set of an isometric involution of $X$. Note that reflective submanifolds are automatically totally geodesic.

\begin{theorem}\label{reflective}
Let $X$ be a Riemannian symmetric space of noncompact type and $Y$ a reflective submanifold of $X$. Let $\Gamma < \mathrm{Isom}(X)$ be a discrete subgroup, and $\Lambda < \Gamma$ be the stabilizer of $Y$ in~$\Gamma$. If $\Lambda$ acts cocompactly on $Y$, then the double $\Gamma \ast_\Lambda \Gamma$ embeds in $\mathsf{GL}_n(\mathbb{R})$ for some dimension $n=n(X)$ depending only on $X$.
\end{theorem}

Note that in Theorem~\ref{reflective} there is no arithmeticity assumption on the ambient discrete group~$\Gamma$, nor is $\Gamma$ even required to be of finite covolume.

Reflective submanifolds of symmetric spaces, which appear elsewhere in the literature under the guise of affine symmetric pairs, have been classified \cite{zbMATH03152734, MR367873, MR587545}. For instance, if $f$ is a nondegenerate quadratic form on $\mathbb{R}^3$, then $(\mathsf{SL}_3(\mathbb{R}), \mathsf{SO}(f; \mathbb{R}))$ is an affine symmetric pair. If $f$ is moreover defined over $\mathbb{Q}$ and is $\mathbb{Q}$-anisotropic, then $\mathsf{SO}(f; \mathbb{Z})$ is cocompact in $\mathsf{SO}(f; \mathbb{R})$ \cite{MR147566, MR141672}, and one deduces from Theorem~\ref{reflective} that in this case the double $\mathsf{SL}_3(\mathbb{Z}) \ast_{\mathsf{SO}(f; \mathbb{Z})}\mathsf{SL}_3(\mathbb{Z})$ is linear. Note that if $f$ is in addition $\mathbb{R}$-isotropic (i.e., indefinite), then it follows from Margulis superrigidity that the image of $\mathsf{SO}(f; \mathbb{Z})$ under no faithful (indeed, infinite-image) finite-dimensional real representation of $\Gamma$ is precompact, so that linearity of this particular double cannot be deduced from Theorem~\ref{compact-1}.

For $\mathbb{K} = \mathbb{R}, \mathbb{C}, \mathbb{H},$ or $\mathbb{O}$, any totally geodesic $\mathbb{K}$-subspace of a $\mathbb{K}$-hyperbolic space is reflective. In particular, any totally geodesic subspace of a real hyperbolic space is reflective, and Theorem~\ref{reflective} is novel even in this setting, though Baker and Cooper \cite{MR2417445} show in this case that if one moreover assumes that the compact subspace $\Lambda \backslash Y \subset \Gamma \backslash X$ is embedded, that the normal bundle of $\Lambda \backslash Y$ in $\Gamma \backslash X$ has trivial holonomy, and that the length of a shortest orthogeodesic to $\Lambda \backslash Y$ in $\Gamma \backslash X$ is sufficiently large, then the double ${\Gamma \ast_\Lambda \Gamma}$ embeds discretely in the isometry group of some real hyperbolic space whose dimension depends only on the dimensions of $X$ and $Y$, and hence only on the dimension of $X$; a version of this statement also follows from work of Danciger--Gu\'eritaud--Kassel~\cite{danciger2024combination}. (The approach of Baker--Cooper should also yield a similar statement to their result in the case that $X$ is a $\mathbb{K}$-hyperbolic space for $\mathbb{K}= \mathbb{R}, \mathbb{C},$ or $\mathbb{H}$, and $Y \subset X$ is a $\mathbb{K}$-subspace, where the target of the output discrete and faithful representation of the double $\Gamma \ast_\Lambda \Gamma$ is now the isometry group of some $\mathbb{K}$-hyperbolic space whose dimension depends only on those of $X$ and $Y$. We thank Nicolas Tholozan for pointing this out to us.) On the other hand, every compact arithmetic real hyperbolic manifold $\Gamma \backslash X$ of simplest type contains nonembedded compact immersed totally geodesic submanifolds $\Lambda \backslash Y$ of each dimension $0 < \dim(Y) < \dim(X)$. Moreover, it follows from work of Agol~\cite{arXiv:math/0612290}, Belolipetsky--Thomson \cite{MR2821431}, and Bergeron--Haglund--Wise \cite{MR2776645} that, for any $\epsilon > 0$, there are closed real hyperbolic manifolds $\Gamma \backslash X$ of arbitrary dimension, at least some of which are nonarithmetic, possessing embedded compact totally geodesic hypersurfaces $\Lambda \backslash Y$ the shortest orthogeodesics to which are of length $< \epsilon$. These examples~$\Gamma$ are known to be virtually compact special~\cite{MR2776645}, so that the doubles $\Gamma \ast_\Lambda \Gamma$ are in this case again virtually compact special and hence linear~\cite{MR4693936, arXiv:2605.21734}, though the approach via specialness provides no control on the dimension of a faithful representation.

As another sample application of Theorem~\ref{reflective}, if $\Gamma \backslash X$ is a Picard modular surface and ${\Lambda \backslash Y \subset \Gamma \backslash X}$ is any compact totally geodesic complex curve (such curves always exist; see~\cite{MR3705536, MR3708964}), then the double $\Gamma \ast_\Lambda \Gamma$ is linear. Note that linearity of such a double cannot be deduced for instance from Theorem~\ref{absolutelysimple}(\ref{absolutelysimple1}).

It is straightforward to check that a {\em maximal} totally complex subspace $Y$ of a quaternionic hyperbolic space $X$ is reflective. In this setting, if $\Gamma \backslash X$ is compact and the projection $\Lambda \backslash Y$  of~$Y$ to $\Gamma \backslash X$ is compact and embedded, then the double $\Gamma \ast_\Lambda \Gamma$, while Gromov-hyperbolic by the combination theorem of Bestvina and Feighn \cite{BF}, fails to embed discretely in any simple Lie group of real rank one, as can be deduced from the superrigidity theorem of Corlette \cite{Corlette} by an argument similar to the proof of \cite[Thm.~1.7]{TT-IMRN}. On the other hand, it follows from Theorem~\ref{reflective} that such doubles are nevertheless linear, and are thus perhaps good candidates for examples of linear Gromov-hyperbolic groups that fail to embed discretely in any connected Lie group. (Linearity of such doubles in fact also follows from Theorem~\ref{absolutelysimple}(\ref{absolutelysimple1}) in light of \cite[Prop.~2.8]{MR4474914}.) It is indeed open whether there is a linear Gromov-hyperbolic group that fails to admit an Anosov representation in the sense of Labourie \cite{MR2221137} and Guichard--Wienhard~\cite{MR2981818}; see~\cite[pg.~5]{arXiv:2510.21334} and~\cite[Rmk.~3.6]{MR4012341}. We remark however that it follows from work~\cite{arXiv:2504.21802} of Dey with the second-named author that, even in the setting of this paragraph, one can always find a finite-index subgroup $\Gamma'<\Gamma$ containing $\Lambda$ such that the double $\Gamma' \ast_\Lambda \Gamma'$ indeed admits an Anosov representation.

\subsection*{Organization of the paper} We postpone the proof of the negative statement (Theorem~\ref{negative}) to Section~\ref{nonlinear-amalgams}. In Section~\ref{motherlemma}, we prove a key result (Theorem~\ref{main}), to which all subsequent positive statements will be reduced, regarding linearity of doubles of matrix groups along certain subgroups for which membership is detected by the top-left entry. Section~\ref{amalgamsalongprecompact} is devoted to the proof of Theorem~\ref{compact-1} and its consequences, while Section~\ref{doublingalongreflective} contains the proof of Theorem~\ref{reflective}. 

\subsection*{Acknowledgements} We thank Nikolay Bogachev, Fanny Kassel, Amir Mohammadi, Hee Oh, Piotr Przytycki, Jean Raimbault, Eduardo Reyes, Matthew Stover, Nicolas Tholozan, and Henry Wilton for helpful discussions. We are also grateful to Changqian Li for sharing with us his alternative proofs of Theorem~\ref{specialoverQC}. SD is supported by Deutsche Forschungsgemeinschaft (DFG, German Research Foundation) under Germany's Excellence Strategy - EXC-2047/1 - 390685813. The final stage of this research was conducted at the Simons Laufer Mathematical Sciences Institute in Berkeley, California, during the Spring 2026 semester (National Science Foundation Grant No. DMS-2424139); the authors are grateful to the organizers of the program Geometry and Dynamics for Discrete Subgroups of Higher Rank Lie Groups, and to the institute for its hospitality and support.

\section{Doubling along subgroups detected by the top-left entry}\label{motherlemma}

This section is devoted to the proof of Theorem~\ref{main}, to which many of the statements in the sequel will be reduced.

First, we establish some notation. Given an integral domain $\mathsf{R}$, we denote by $\mathsf{Mat}_{m}(\mathsf{R})$ (respectively, by $\mathsf{Mat}_{m\times r}(\mathsf{R})$) the free $\mathsf{R}$-module of $m\times m$ (resp., $m\times r$) matrices with entries in~$\mathsf{R}$. For $1\leq i,j \leq m$, we will denote by $E_{ij}(m)\in \mathsf{Mat}_m(\mathsf{R})$ the matrix whose $(i,j)$-entry  is~$1$ and the rest of whose entries are zero. For $d,k \in \mathbb{N}$, we view the product $\mathsf{GL}_k(\mathsf{R})\times \mathsf{GL}_d(\mathsf{R})$ as a subgroup of $\mathsf{GL}_{k+d}(\mathsf{R})$ via the block-diagonal embedding $(A,B) \mapsto \begin{pmatrix}[0.6]
A & \\ 
 & B
\end{pmatrix}$. For a matrix $g=(g_{ij})_{i,j=1}^{n}$ in $\mathsf{GL}_n(\mathsf{R})$, we define the {\em top-left $k\times k$ block of} $g$ to be the matrix $(g_{ij})_{i,j=1}^{k}$.

\begin{theorem}\label{main} Let $\mathsf{R}$ be an integral domain, and $\Gamma_1, \Gamma_2$ be subgroups of $ \mathsf{GL}_n(\mathsf{R})$, $n \geq 2$.
\begin{enumerate}

\item \label{main-item1} Suppose that $\Lambda < \Gamma_1 \cap \Gamma_2$ is a subgroup of $\{1\}\times \mathsf{GL}_{n-1}(\mathsf{R})$ such that for $i=1,2$, a matrix $\gamma \in \Gamma_i$ lies in $\Lambda$ if and only if the top-left entry of $\gamma$ is $1$. Then the amalgam $\Gamma_1 \ast_\Lambda \Gamma_2$ admits a faithful representation into $\mathsf{GL}_{n+1}(\mathsf{R}[t])$, where $\mathsf{R}[t]$ denotes the ring of polynomials over $\mathsf{R}$.

\item \label{main-item2} Suppose that $\Lambda< \Gamma_1 \cap \Gamma_2$ is a subgroup of $\{I_k\}\times \mathsf{GL}_{n-k}(\mathsf{R})$, $2\leq k\leq n-1$, such that for $i =1,2$, a matrix $\gamma \in \Gamma_i$ lies in $\Lambda$ if and only if the top-left $k\times k$ block of $\gamma$ is the identity. Then the amalgam $\Gamma_1\ast_{\Lambda} \Gamma_{2}$ admits a faithful representation into $\mathsf{GL}_{n^2+1}(\mathsf{R}[s,t])$,  where $\mathsf{R}[s,t]$ denotes the ring of polynomials over $\mathsf{R}$ in two variables.
\end{enumerate}
\end{theorem}

\begin{proof}[Proof of Theorem \ref{main}(\ref{main-item1})] View $\mathsf{GL}_n(\mathsf{R})$ as the subgroup $\{1\} \times \mathsf{GL}_n(\mathsf{R}) < \mathsf{GL}_{n+1}(\mathsf{R}[t])$. Consider the matrices $u_t, \tau \in \mathsf{GL}_{n+1}(\mathsf{R}[t])$ given by

\[
u_t := \begin{pmatrix} 1 & t & \\ 0  & 1 &  \\  &  & I_{n-1} \end{pmatrix}, \> 
\tau : = \begin{pmatrix} 0 & 1 & \\  1 & 0 &  \\  &  & I_{n-1} \end{pmatrix},
\]
and let $\rho_t :\Gamma_1 \ast_\Lambda \Gamma_2 \rightarrow \mathsf{GL}_{n+1}(\mathsf{R}[t])$ be the representation induced by $\gamma_1 \mapsto u_t \gamma_1 u_t^{-1}$ on $\Gamma_1$ and $\gamma_2 \mapsto \tau u_t \gamma_2 u_t^{-1} \tau$ on $\Gamma_2$.

Say a matrix $M = (m_{ij})_{i,j = 1}^{n+1} \in \mathsf{Mat}_{n+1}(\mathsf{R}[t])$ is {\em sufficiently transcendental} if $\deg(m_{ij}) \leq \deg(m_{11})$ for $1 \leq i,j \leq n+1$ and $\deg(m_{ij}) < \deg(m_{11})$ for $i > 1$. Note that a product of two sufficiently transcendental matrices is sufficiently transcendental. Note also that for $i=1,2$, and $\gamma_i \in \Gamma_i \smallsetminus \Lambda$, the product $(u_t \gamma_1 u_t^{-1})(\tau u_t \gamma_2 u_t^{-1}\tau)$ is sufficiently transcendental; indeed, a direct computation shows that the top-left entry of $(u_t \gamma_1 u_t^{-1})(\tau u_t \gamma_2 u_t^{-1}\tau)$ is a degree-$2$ polynomial in $t$ whose leading coefficient is $\big((\gamma_1)_{11}-1\big)\big((\gamma_2)_{11}-1\big)$, while the remaining entries of $(u_t \gamma_1 u_t^{-1})(\tau u_t \gamma_2 u_t^{-1}\tau)$ (respectively, the entries of $(u_t \gamma_1 u_t^{-1})(\tau u_t \gamma_2 u_t^{-1}\tau)$ below the first column) are of degree $\leq 2$ (resp., are of degree $\leq 1$). It thus follows that for each element $\gamma \in \Gamma_1 \ast_\Lambda \Gamma_2$ of the form
\begin{equation*}\label{normalform}
\gamma = \gamma_1^{(1)}\gamma_2^{(1)}\cdots\gamma_1^{(r)}\gamma_2^{(r)},
\end{equation*}
where $\gamma_i^{(1)}, \ldots, \gamma_i^{(r)} \in \Gamma_i \smallsetminus \Lambda$ for $i=1,2$, the matrix $\rho_t(\gamma)$ is sufficiently transcendental, and hence different from the identity. Since every element of $\Gamma_1 \ast_\Lambda \Gamma_2$ is either conjugate into one of the factors or conjugate to an alternating product of even length (see, for instance, \cite[\S1.2,~Thm.~1]{MR607504}), we conclude that the representation $\rho_t$ is faithful.
\end{proof}

\begin{proof}[Proof of Theorem \ref{main} (\ref{main-item2})]

Consider the left translation representation $L:\mathsf{GL}_n(\mathsf{R}[s])\rightarrow \mathsf{GL}(\mathsf{Mat}_{n}(\mathsf{R}[s]))$, $$L(h)M=hM, \ M\in \mathsf{Mat}_n(\mathsf{R}[s]), h \in \mathsf{GL}_n(\mathsf{R}[s]).$$ Let $W,W'\subset \mathsf{Mat}_n(\mathsf{R}[s])$ be the complementary free $\mathsf{R}[s]$-submodules given by  \begin{align*} 
W &:=\Big\{\begin{pmatrix}[0.6]
X &  \\ 
 & 0
\end{pmatrix}: X\in \mathsf{Mat}_{k}(\mathsf{R}[s]) \Big\},\\  W'&:=\Big\{\begin{pmatrix}[0.6]
0 & \ast \\ 
\ast & Y
\end{pmatrix}: Y \in \mathsf{Mat}_{n-k}(\mathsf{R}[s]) \Big\}.\end{align*} Fix the ordered basis $\mathcal{B}_{1}:=\big(\mathcal{E}, E_{22}(n),\ldots, E_{kk}(n), E_{12}(n),\ldots,E_{k(k-1)}(n)\big)$ for $W$, where $\mathcal{E}:=\begin{pmatrix}[0.8]
I_{k} &  \\ 
 & 0
\end{pmatrix}$, and some ordered basis $\mathcal{B}_{2}$ for $W'$. Consider the unipotent element $g_{s}\in \mathsf{GL}(\mathsf{Mat}_n(\mathsf{R}[s]))$, $$g_{s}:=I_{n^2}+\sum_{j=1}^{k^2-1}s^{j}E_{1j}(k^2),$$ and the representation $L_{s}:\mathsf{GL}_n(\mathsf{R}[s]) \rightarrow \mathsf{GL}( \mathsf{Mat}_n(\mathsf{R}[s]))$ given by 
\begin{align*}L_{s}(h):=g_{s} L(h)g_{s}^{-1}, \ \  h \in \mathsf{GL}_n(\mathsf{R}[s]).\end{align*} For $h \in \mathsf{GL}_n(\mathsf{R})$, we may write $$L(h)\mathcal{E}=h_{11}\mathcal{E}+\sum_{i=2}^k(h_{ii}-h_{11})E_{ii}(n)+\sum_{i\neq j}h_{ij}E_{ij}(n)+M_{h}$$ for some $M_{h}\in W'$. In particular, with respect to the ordered basis $(\mathcal{B}_1,\mathcal{B}_2)$ for $\mathsf{Mat}_n(\mathsf{R}[s])$, the top-left entry of the matrix $L_{s}(h)\in \mathsf{GL}(\mathsf{Mat}_n(\mathsf{R}[s]))$ is the polynomial \begin{align*} \big( L_{s}(h)\big)_{11}=h_{11}+\sum_{i=2}^{k}\big(h_{ii}-h_{11}\big)s^{i-1}+h_{12}s^k+\cdots+h_{k(k-1)}s^{k^2-1}.\end{align*} 
Hence, for $h \in \mathsf{GL}_n(\mathsf{R})$, the top-left entry $\big(L_{s}(\gamma)\big)_{11}=1$ if and only if the top-left $k\times k$ block of $h$ is the identity. By assumption, for $\gamma \in \Gamma_i$, $i=1,2$, the latter holds if and only if $\gamma \in \Lambda$. Moreover, on $\Lambda$, the representations $L$ and $L_s$ coincide, act as the identity on $W$, and preserve the subspace~$W'$. Therefore, by applying Theorem \ref{main} (\ref {main-item1}) for the groups $L_{s}(\Gamma_{1})$ and~$L_{s}(\Gamma_2)$ and their common subgroup $L_{s}(\Lambda)$, we obtain a faithful representation  $\rho:\Gamma_1 \ast_{\Lambda} \Gamma_{2} \rightarrow \mathsf{GL}_{n^2+1}(\mathsf{R}[s,t])$. This completes the proof of the theorem. \end{proof}

\begin{Exmp}
\normalfont{
Let $K$ be a number field whose stufe is $4$, so that $K$ is in particular totally complex, and denote by $\mathcal{O}_K$ the ring of integers of $K$. Then, since the quadratic form ${x_1^2+x_2^2+x_3^2}$ is anisotropic over $K$, the Borel--Harish-Chandra theorem \cite{MR147566,MR141672} implies that the group $\Gamma := \mathsf{SO}_3(\mathcal{O}_K)$ embeds as an irreducible cocompact lattice in a product of $d/2$ copies of $\mathsf{SO}_3(\mathbb{C}) \cong \mathsf{PSL}_2(\mathbb{C})$. Moreover, it follows immediately from Theorem~\ref{main}(\ref{main-item1}) that the double $\Gamma \ast_\Lambda \Gamma$ is linear if $\Lambda$ is taken to be the stabilizer in $\Gamma$ of the standard basis vector $e_1$. Note that for $d>2$ this is an example to which Theorem~\ref{compact-1} will not apply, since it is indeed not difficult to verify via Margulis superrigidity \cite{Mar} that for such $d$ the image of $\Lambda$ under no faithful finite-dimensional real representation of $\Gamma$ is precompact. For an example with $d>2$, one can for instance take $K =\mathbb{Q}(\sqrt{-7}, \sqrt{17})$. Indeed, the $2$-adics $\mathbb{Q}_2$ contain square roots of $-7$ and $17$ by Hensel's lemma, and the stufe of $\mathbb{Q}_2$ is $4$.
}
\end{Exmp}

\section{Doubling along intersections with compact subgroups}\label{amalgamsalongprecompact}

In this section we prove Theorem \ref{compact-1} and establish some consequences. We first deduce from Theorem~\ref{main} the following technical lemma, which will also be useful in the following section.

\begin{lemma} \label{multiple} Let $\Gamma_{1},\Gamma_2$ be groups with a common subgroup $\Lambda$. Suppose we are given for some integer $p \geq2$, for $i=1,\ldots,p$, for $j=1,2$, and for some integral domain $\mathsf{R}$ representations  \hbox{$\rho_{i}^j:\Gamma_{j}\rightarrow \mathsf{GL}_{n_i}(\mathsf{R})$}, at least one of which is faithful, satisfying the following properties:
\begin{itemize}
\item $\rho_{i}^j(\Lambda)\subset \{1\}\times \mathsf{GL}_{n_i-1}(\mathsf{R})$;
\item the representations $\rho_i^1$ and $\rho_i^2$ coincide on $\Lambda$; and
\item if $\gamma \in \Gamma_{j}$ and the top-left entry of each of the $\rho_i^j(\gamma)$ is equal to $1$, then $\gamma \in \Lambda$.\end{itemize}
\noindent Then there is a faithful representation of  $\Gamma_1 \ast_{\Lambda}\Gamma_{2}$ into $\mathsf{GL}_{r}(\mathsf{R}[s,t])$, where \hbox{$r=(\sum_{i=1}^p n_i)^2+1$.}\end{lemma}

\begin{proof} For $j=1,2$, let $\times_{i=1}^{p}\rho_{i}^j:\Gamma_{j}\rightarrow \mathsf{GL}_{d}(\mathsf{R})$ be the $p$-fold product of the representations $\rho_{1}^j,\ldots ,\rho_{p}^j$, where $d=\sum_{i=1}^p n_i$. For $i=1,2$, let $(e_{i1},\ldots, e_{in_i})$ be the standard basis of $\mathsf{R}^{n_i}$. By reordering the basis $(e_{11},\ldots,e_{1n_1}, \ldots, e_{p1},\ldots,e_{pn_p})$ of $\bigoplus_{i=1}^p \mathsf{R}^{n_i}$ as $(e_{11},\ldots, e_{p1}, \ldots, e_{p2},\ldots,e_{pn_p})$, we obtain a faithful represention $\psi_{j}:\Gamma_{j} \rightarrow \mathsf{GL}_{d}(\mathsf{R})$, conjugate to $\times_{i=1}^{p}\rho_{i}^j$, with the property that for any $\gamma \in \Gamma_j$ the top left $p\times p$ block of $\psi_j(\gamma)$ is a diagonal matrix whose diagonal entries are the top-left entries of the $\rho_i^j(\gamma)$. Thus, by assumption, if $\gamma\in \Gamma_{j}$ and the top left $p\times p$ block of $\psi_j(\gamma)$ is the identity matrix, then $\gamma \in \Lambda$. Therefore, by applying Theorem \ref{main}(\ref{main-item2}) for the groups $\psi_{1}(\Gamma_{1}), \psi_2(\Gamma_2)$ and their common subgroup $\psi_{1}(\Lambda)$, we obtain a faithful representation of $\Gamma_1 \ast_{\Lambda}\Gamma_{2}$ into $\mathsf{GL}_r(\mathsf{R}[s,t])$, where $r=d^2+1$.\end{proof}

For an integral domain $\mathsf{R}$, we denote by $\textup{Sym}(\mathsf{R})$ the free $\mathsf{R}$-module of $r \times r$ symmetric matrices over $\mathsf{R}$, and by $\textup{Sym}_r:\mathsf{SL}_{r}(\mathsf{R})\rightarrow \mathsf{SL}(\textup{Sym}(\mathsf{R}))$ the representation given by 
\[
\textup{Sym}_r(g)M:=gMg^t
\]
for $M \in\textup{Sym}(\mathsf{R})$. We first establish the following special case of Theorem \ref{compact-1}, where the subgroup over which one doubles is the stabilizer of a line in the special orthogonal group. 

\begin{theorem}\label{compact-1'} Let $\mathsf{R}$ be a subring of $\mathbb{R}$. Let $\mathsf{C}$ be the stabilizer of the standard basis vector $e_1$ in $\mathsf{SO}(n)$, $n\geq 2$, and let $\Lambda:= \mathsf{C} \cap \mathsf{SL}_n(\mathsf{R})$. Then there is a faithful representation of the amalgam $\mathsf{SL}_n(\mathsf{R}) \ast_{\Lambda}\mathsf{SL}_n(\mathsf{R})$ into $\mathsf{SL}_{d}(\mathsf{R}[s,t])$, where $d=\big(n^2+3n+2\big)^2+1$.\end{theorem}

\begin{proof} Consider the embedding $ \iota:\mathsf{SL}_n(\mathsf{R})\xhookrightarrow{}\mathsf{SL}_{n+1}(\mathsf{R}),\  \iota (g)=
\begin{pmatrix}[0.8]
g &  \\ 
  & 1
\end{pmatrix}$, and fix ordered bases $\mathcal{B}_{n}$ and $\mathcal{B}_n'$ for $\textup{Sym}(\mathsf{R}^{n+1})$ such that
\begin{itemize}
\item the first basis vector in $\mathcal{B}_{n}$ (respectively, in $\mathcal{B}_n'$) is the identity matrix $I_{n+1}$ (resp., is the matrix $E_{11}(n+1)$), and
\item the remaining basis vectors in $\mathcal{B}_{n}$ (respectively, in $\mathcal{B}_n'$) lie in the $\mathsf{R}$-submodule $V_n:=\{M \in \textup{Sym}(\mathbb{R}^{n+1}):\textup{tr}(M)=0\}$ (resp., in \hbox{$V_n':=\{M \in \textup{Sym}(\mathbb{R}^{n+1}):\textup{tr}(ME_{11}(n+1))=0\}$).}
\end{itemize}
With respect to the bases $\mathcal{B}_n, \mathcal{B}_{n}'$, we obtain two faithful representations $\rho_n, \rho_n': \mathsf{SL}_{n}(\mathsf{R})\rightarrow \mathsf{SL}_{r}(\mathsf{R})$ conjugate to the representation $\textup{Sym}_{n+1}\circ \iota:\mathsf{SL}_n(\mathsf{R})\rightarrow \mathsf{SL}(\textup{Sym}(\mathsf{R}^{n+1}))$, where $r=\frac{(n+2)(n+1)}{2}$, satisfying for any $g\in \mathsf{SL}_n(\mathsf{R})$, \begin{align}\label{trace-equation}\big(\rho_n(g)\big)_{11}&=\frac{\textup{tr}(gg^t)+1}{n+1},\\ \big(\rho_n'(g)\big)_{11}&=(g_{11})^2. \end{align} Moreover, observe that if $g\in \Lambda$, then $\textup{Sym}_{n+1}(\iota(g))$ fixes both $I_{n+1},E_{11}(n+1)\in \textup{Sym}(\mathsf{R}^{n+1})$ and preserves the hyperplanes $V_n$ and $V_n'$. In other words, $\rho_n(\Lambda)$ and $\rho_n'(\Lambda)$ are subgroups of $\{1\}\times \mathsf{SL}_{r-1}(\mathsf{R}) \subset \mathsf{SL}_r(\mathsf{R})$.

\par Now we check that  $\rho_{n}$ and $\rho_n'$ satisfy the assumptions of Lemma \ref{multiple}. To check that, suppose $g\in \mathsf{SL}_n(\mathsf{R})$ satisfies $$\big(\rho_n(g)\big)_{11}=\big(\rho_n'(g)\big)_{11}=1.$$ By (\ref{trace-equation}), we have $\textup{tr}(gg^t)=n$, and since $gg^t$ is positive-definite and $\textup{det}(gg^t)=1$, we conclude that $gg^t=I_{n}$. In addition, since $(g_{11})^2=1$ and the columns of $g$ form an orthonormal basis of $\mathbb{R}^n$, we also have $g\in \mathsf{C}$. Thus, by applying Lemma \ref{multiple} for $\rho_n$ and $\rho_n'$, we obtain a faithful representation of $\mathsf{SL}_n(\mathsf{R}) \ast_{\Lambda}\mathsf{SL}_n(\mathsf{R})$ into $\mathsf{SL}_{d}(\mathsf{R}[s,t])$, where $d=(2r)^2+1$. \end{proof}

Now let $K$ be a subfield of $\mathbb{R}$, and for $n,r\in \mathbb{N}$, let $V_{n,r}\subset K[x_{11},\ldots,x_{nn}]$ denote the vector subspace of polynomials over $K$ of degree at most $r$. To prove Theorem~\ref{compact-1}, we will use the following proposition due to Chevalley; see also \cite[11.2, Ch. IV]{Humphreys}.

\begin{proposition}\label{Chevalley-embedding} Let $\mathsf{H}<\mathsf{SL}_n(K)$ be a $K$-subgroup defined by polynomials of degree at most~$r$ in $K[x_{11},\ldots,x_{nn}]$. There is a faithful $K$-polynomial representation \hbox{$\rho_{\mathsf{H}}:\mathsf{SL}_n(K)\rightarrow \mathsf{SL}(V_{n,r})$} and a $K$-subspace $V_{\mathsf{H}}\subset V_{n,r}$ such that $\mathsf{H}=\big\{g\in \mathsf{SL}_n(K): \rho_{\mathsf{H}}(g)V_{\mathsf{H}}=V_{\mathsf{H}}\big\}$.\end{proposition}

\begin{proof}[Proof of Theorem \ref{compact-1}] By Proposition \ref{Chevalley-embedding}, there is a faithful $K$-polynomial representation $\rho_{\mathsf{C}}:\mathsf{SL}_{n}(K)\rightarrow \mathsf{SL}_{d_n}(K)$ 
satisfying
\[
\Lambda =\big \{g\in \mathsf{SL}_{n}(K): \rho_{\mathsf{C}}(g)\textup{span}_K(e_1,\ldots,e_s)=\textup{span}_K(e_1,\ldots,e_s)\big \}
\]
and $\rho_{\mathsf{C}}(\Lambda)\subset \mathsf{O}(d_n; K)$, where $d_n:=\textup{dim}(V_{n,r})$ and $1 \leq s \leq d_n-1$. By composing $\rho_{\mathsf{C}}$ first with the embedding $\mathsf{SL}(K^{d_n})\xhookrightarrow{} \mathsf{SL}(K^{d_n})\times \{1\}$ and then with $\wedge^s:\mathsf{SL}(K^{d_n+1})\rightarrow \mathsf{SL}(\wedge^{s}K^{d_n+1})$, we obtain a faithful representation $\rho_{\mathsf{C},s}:\mathsf{SL}_n(K)\rightarrow \mathsf{SL}(\wedge^s K^{d_n+1})$ satisfying $$\rho_{\mathsf{C},s}(\Lambda)=\rho_{\mathsf{C},s}(\mathsf{SL}_n(K))\cap \big(\{\pm 1\}\times \mathsf{O}(W;K)\big),$$ where $W:=\textup{span}_K\big(\{e_{i_1}\wedge\cdots \wedge e_{i_s}: \{i_1,\ldots i_s\}\neq \{1,\ldots,s\}\} \big).$ Therefore, there is an embedding 
\[
\mathsf{SL}_n(K)\ast_\Lambda \mathsf{SL}_n(K)\xhookrightarrow{}   \mathsf{SL}(\wedge^s K^{d_n+1}) \ast_{\Lambda'} \mathsf{SL}(\wedge^s K^{d_n+1}),
\]
where $\Lambda':=\mathsf{SL}(\wedge^s K^{d_n}) \cap (\{\pm 1\}\times \mathsf{O}(W;K)) $. Since $\binom{d_n+1}{\lceil\frac{d_n+1}{2}\rceil}\geq \binom{d_n}{s}$, by Theorem~\ref{compact-1'} there is a faithful representation $\varphi_{\mathsf{C}}:\mathsf{SL}_n(K) \ast_\Lambda \mathsf{SL}_n(K)\rightarrow \mathsf{SL}_q(K[s,t])$ with $q=((d_n')^2+3d_n'+2)^2+1$, $d_n'=\binom{d_n+1}{\lceil \frac{d_n+1}{2} \rceil}$ and $d_n=\textup{dim}(V_{n,r})=\binom{n^2+r}{r}-1$. Finally, note that $d_n'\leq 2^{d_n}$ and hence $q\leq 2^{4\binom{n^2+r}{r}-3}$. \end{proof}

\begin{proof}[Proof~of~Corollary~\ref{compactcorollaries}(\ref{compact-2})]
We first claim there is a finite-index subgroup $\Gamma_\Lambda$ of $\Gamma$ such that $\Gamma_\Lambda \cap \mathsf{L} = \Lambda$. Indeed, let $1, \gamma_1, \ldots, \gamma_r$ be a full set of representatives of $(\Gamma \cap \mathsf{L})/\Lambda$. We may then take $\Gamma_\Lambda$ to be any finite-index subgroup of $\Gamma$ containing $\Lambda$ but excluding $\gamma_1, \ldots, \gamma_r$; the existence of such a finite-index subgroup of $\Gamma$ is guaranteed by separability of $\Lambda$ in $\Gamma$.

Viewing $V$ now as a representation of $\Gamma_\Lambda$, let $\rho: \Gamma \rightarrow \mathsf{GL}(W)$ be the representation induced on~$\Gamma$, so that $V$ is a sub-$\Gamma_\Lambda$-representation of $W$. Let $\mathsf{C}$ be the Zariski-closure of~$\rho(\Lambda)$ in~$\mathsf{GL}(W)$, and let $\mathsf{H}$ be the algebraic subgroup of $\mathsf{GL}(W)$ consisting of those elements preserving~$V$ and whose restriction to $V$ belongs to $\mathsf{L}$. Since $\mathsf{L} \subset \mathsf{H}$ and $\rho(\Gamma) \cap \mathsf{H} = \rho(\Lambda)$, we have $\rho(\Gamma) \cap \mathsf{L} = \rho(\Lambda)$. Moreover, from the description of the induced representation, one sees that since $\Lambda$ preserves a metric on $V$, the same is true of $\rho(\Lambda)$ on $W$, so that $\mathsf{L}$ is compact. The conclusion of Corollary \ref{compactcorollaries} (\ref{compact-2}) now follows from Theorem~\ref{compact-1}, together with the fact that, for any subfield $K \subset \mathbb{R}$, the domain $K[s,t]$ embeds in $\mathbb{C}$, so that any group that is $K[s,t]$-linear is $\mathbb{C}$-linear and thus $\mathbb{R}$-linear via restriction of scalars.
\end{proof}

We remark that embeddability of $K[s,t]$ in $\mathbb{C}$ is clear if $K$ is, for instance, countable, but indeed any characteristic-zero domain of cardinality at most that of $\mathbb{C}$ embeds in $\mathbb{C}$.

\begin{proof}[Proof of Corollary \ref{compactcorollaries}(\ref{precompactgaloisconjugate})] We may extend the embedding $\sigma:K \rightarrow \mathbb{C}$ to an automorphism of $\mathbb{C}$, which we continue to denote by $\sigma$. We have assumed that $\mathsf{G} = {\bf G}(\mathbb{C})$ for some $\mathbb{C}$-subgroup ${\bf G} < \mathsf{SL}_n$, and we will write $\mathsf{G}^\sigma := {\bf G}^\sigma(\mathbb{C})$. Observe that $\mathsf{G}^{\sigma}$ is a $\mathbb{C}$-algebraic subgroup and $\Gamma^{\sigma}\cap \mathsf{G}^{\sigma}=(\Gamma\cap \mathsf{G})^{\sigma}$ is by assumption precompact in $\mathsf{GL}_n(\mathbb{C})$.
Let $j: \mathsf{GL}_{n}(\mathbb{C})\rightarrow \mathsf{GL}_{2n}(\mathbb{R})$ be the standard embedding given by restriction of scalars, i.e., given by $g \mapsto \begin{pmatrix}\textup{Re}(g) & -\textup{Im}(g)\\ \textup{Im}(g) & \textup{Re}(g)\end{pmatrix}$ for $g\in \mathsf{GL}_n(\mathbb{C})$. Observe moreover that, since $j(\mathsf{G}^\sigma)$ is an $\mathbb{R}$-algebraic subgroup of $\mathsf{GL}_{2n}(\mathbb{R})$, we have that $j(\Lambda^{\sigma})$ is of finite index in $j(\Gamma^{\sigma})\cap \mathsf{L}$, where $\mathsf{L}$ denotes the (Zariski-)closure of $j(\Gamma^{\sigma})\cap j(\mathsf{G}^{\sigma})$ in $\mathsf{GL}_{2n}(\mathbb{R})$. Since $\mathsf{H}$ is compact and $\Lambda^{\sigma}$ is separable in $\Gamma^{\sigma}$, we conclude from Corollary \ref{compact-2} that the double $\Gamma^{\sigma}\ast_{\Lambda^{\sigma}}\Gamma^{\sigma}$ is linear over $\mathbb{R}$. The latter double is isomorphic to the double $\Gamma\ast_{\Lambda}\Gamma$ and the corollary follows.\end{proof}

The following lemmas will be used to establish Corollary 
\ref{retract}, which will in turn be used to prove Theorem~\ref{specialoverQC}. These arguments can all be found in \cite{MR3663601, 2024arXiv240317964K}, but we include them here for the convenience of the reader.

\begin{lemma}\label{cosets}
Let $\Gamma < \mathsf{GL}_d(\mathbb{R})$ and $\Lambda < \Gamma$. Let $\Gamma' < \Gamma$ be of finite index, and set $\Lambda' := \Gamma' \cap \Lambda$. If $\Gamma' \cap \overline{\Lambda'}^\mathrm{Zar} = \Lambda'$, then $\Gamma \cap \overline{\Lambda}^\mathrm{Zar} = \Lambda$. 
\end{lemma}

\begin{proof}
Let $\lambda_1, \ldots, \lambda_m \in \Lambda$ be a set of left coset representatives for $\Lambda'$ in $\Lambda$. Then $\overline{\Lambda}^\mathrm{Zar} = \lambda_1 \overline{\Lambda'}^\mathrm{Zar} \cup \cdots \cup \lambda_m \overline{\Lambda'}^\mathrm{Zar}$. Thus, given $\gamma \in \Gamma \cap \overline{\Lambda}^\mathrm{Zar}$, we have that $\gamma \in \lambda \overline{\Lambda'}^\mathrm{Zar}$ for some $\lambda \in \Lambda$. But then $\lambda^{-1} \gamma \in \Gamma \cap \overline{\Lambda'}^\mathrm{Zar} \subset \Gamma' \cap \overline{\Lambda'}^\mathrm{Zar}$, and so $\lambda^{-1}\gamma \in \Lambda'$. We conclude that $\gamma \in \lambda \Lambda' \subset \Lambda$. 
\end{proof}

\begin{lemma}\label{retractinduced}
Let $\Gamma'$ be a finite-index normal subgroup of a group $\Gamma$, and $\rho': \Gamma' \rightarrow \mathsf{GL}_n(\mathbb{R})$ be a representation. Suppose $\Lambda' < \Gamma'$ is a subgroup satisfying $\rho'(\Gamma') \cap \overline{\rho'(\Lambda')}^\mathrm{Zar} = \rho'(\Lambda')$. Then for any subgroup $\Lambda < \Gamma$ such that $\Gamma' \cap \Lambda = \Lambda'$, the representation $\mathrm{ind}_{\Gamma'}$ of $\Gamma$ induced by $\rho'$ satisfies $\mathrm{ind}_{\Gamma'}(\Gamma) \cap \overline{\mathrm{ind}_{\Gamma'}(\Lambda)}^\mathrm{Zar} = \mathrm{ind}_{\Gamma'}(\Lambda)$.
\end{lemma}

\begin{proof}
By Lemma \ref{cosets}, it suffices to show that $\mathrm{ind}_{\Gamma'}(\Gamma') \cap \overline{\mathrm{ind}_{\Gamma'}(\Lambda')}^\mathrm{Zar} = \mathrm{ind}_{\Gamma'}(\Lambda')$, but the latter is true since the restriction $\mathrm{ind}_{\Gamma'} \bigr|_{\Gamma'}$ extends to an ($\mathbb{R}$-algebraic) representation of $\mathsf{SL}_n(\mathbb{R})$.
\end{proof}

One says a subgroup $\Lambda$ of a group $\Gamma$ is a {\em virtual retract} of $\Gamma$ if there is a finite-index subgroup $\Gamma' < \Gamma$ containing $\Lambda$ such that $\Lambda$ is a retract of $\Gamma'$. We will say a subgroup $\Lambda < \Gamma$ is {\em virtually a virtual retract} of $\Gamma$ if there is a finite-index subgroup $\Gamma'$ of $\Gamma$ such that $\Lambda \cap \Gamma'$ is a retract of $\Gamma'$.

\begin{lemma}\label{retractsarelagebraic}
Let $\Gamma < \mathsf{SL}_n(\mathbb{R})$ be a subgroup, and suppose $\Lambda < \Gamma$ is virtually a virtual retract of~$\Gamma$. Then there is a faithful representation $\rho: \Gamma \rightarrow \mathsf{SL}_d(\mathbb{R})$ for some $d \in \mathbb{N}$ such that $\rho(\Gamma) \cap \overline{\rho(\Lambda)}^\mathrm{Zar} = \rho(\Lambda)$. Moreover, if $\Lambda$ is precompact in $\mathsf{SL}_n(\mathbb{R})$, then $\rho$ can be chosen such that $\rho(\Lambda)$ is precompact in $\mathsf{SL}_d(\mathbb{R})$. 
\end{lemma}

\begin{proof}
By assumption, there is a finite-index subgroup $\Gamma' < \Gamma$ and a map $r: \Gamma' \rightarrow \Lambda' := \Gamma' \cap \Lambda$ such that $r\bigr|_{\Lambda'} = \mathrm{Id}_{\Lambda'}$. Up to replacing $\Gamma'$ with $\bigcap_{\gamma \in \Gamma} \gamma \Gamma' \gamma^{-1}$ (and $\Lambda'$ with $\Lambda \cap \bigcap_{\gamma \in \Gamma} \gamma \Gamma' \gamma^{-1}$), we can assume $\Gamma'$ is normal in $\Gamma$. Now let $\rho': \Gamma' \rightarrow \mathsf{SL}_{2n}(\mathbb{R})$ be the representation given by
\[
\gamma \mapsto \begin{pmatrix}  
\gamma & \\
& r(\gamma)
\end{pmatrix}
\]
for $\gamma \in \Gamma'$. Then $\rho'$ satisfies $\rho'(\Gamma') \cap \overline{\rho'(\Lambda')}^\mathrm{Zar} = \rho'(\Lambda')$, since indeed $\rho'(\Gamma') \cap \mathsf{\Lambda}(\mathsf{SL}_n(\mathbb{R})) = \rho'(\Lambda')$, where $\mathsf{\Lambda} : \mathsf{SL}_n(\mathbb{R}) \rightarrow \mathsf{SL}_{2n}(\mathbb{R})$ is the diagonal embedding $g \mapsto \mathrm{diag}(g,g)$. Hence, by Lemma \ref{retractinduced}, we may take $\rho$ to be the representation $\mathrm{ind}_{\Gamma'}$ of $\Gamma$ induced by $\rho'$. 

Moreover, if $\Lambda$ is contained in a compact subgroup $\mathsf{C}< \mathsf{SL}_n(\mathbb{R})$, then $\rho'(\Lambda') \subset \mathsf{\Lambda}(\mathsf{C})$ is precompact, and hence so is $\rho(\Lambda')$ since $\rho\bigr|_{\Gamma'}$ extends to a representation of $\mathsf{SL}_n(\mathbb{R})$. We conclude in this case that $\rho(\Lambda)$ is also precompact since $\Lambda'$ is of finite index in $\Lambda$.
\end{proof}

The following is now immediate from Theorem \ref{compact-1} and Lemma \ref{retractsarelagebraic}.

\begin{corollary} \label{retract} Let $\Gamma<\mathsf{SL}_n(\mathbb{R})$ be a subgroup.
Suppose that $\Lambda<\Gamma$ is virtually a virtual retract of $\Gamma$, and is precompact in $\mathsf{SL}_n(\mathbb{R})$. Then the double $\Gamma \ast_{\Lambda}\Gamma$ is linear over $\mathbb{R}$.\end{corollary}

We are now ready to prove the following more general fact than Theorem~\ref{specialoverQC} (see \cite{MR2684983} for details on the notion of a relatively quasiconvex subgroup of a relatively hyperbolic group). Indeed, Theorem~\ref{specialoverQC} is the particular case of Theorem~\ref{virtuallyspecialrelativelyhyperbolic} where the peripheral subgroups of $\Gamma$ are trivial.

\begin{theorem}\label{virtuallyspecialrelativelyhyperbolic}
Let $\Gamma$ be a virtually compact special group that is hyperbolic relative to a collection of finitely generated virtually abelian subgroups, and $\Lambda < \Gamma$ be a relatively quasiconvex subgroup. Then the double $\Gamma \ast_\Lambda \Gamma$ is linear over $\mathbb{R}$. 
\end{theorem}

\begin{proof}
By work of Agol \cite{agol2018hyperbolic}, there is an embedding of $\Gamma$ in a compact Lie group. Theorem~\ref{specialoverQC} now follows from Corollary \ref{retract}, together with the fact that a residually finite group that is hyperbolic relative to a collection $\mathscr{P}$ of finitely generated virtually abelian subgroups possesses a finite-index subgroup whose intersection with each member of $\mathscr{P}$ is abelian, and the fact that relatively quasiconvex subgroups of compact special groups that are hyperbolic relative to a collection of finitely generated abelian subgroups are virtual retracts; see \cite[Thm.~7.3]{HW08} and \cite[Thm.~1.3]{MR2980525}.
\end{proof}

A similar argument where one instead uses that virtually abelian subgroups of arbitrary virtually special groups are virtual retracts \cite{MR4307692}, yields the following.

\begin{corollary}\label{VSoverVA}
Let $\Gamma$ be a finitely generated virtually special group and $\Lambda < \Gamma$ a virtually abelian subgroup. Then the double $\Gamma \ast_{\Lambda} \Gamma$ is linear over $\mathbb{R}$.
\end{corollary}

Using deep results about Kleinian groups, we also deduce the following.

\begin{corollary}\label{kleinian}
Let $\Gamma$ be a finitely generated discrete subgroup of $\mathrm{Isom}(\mathbb{H}^3)$ and $\Lambda < \Gamma$ an arbitrary finitely generated subgroup. Then the double $\Gamma \ast_\Lambda \Gamma$ is linear.
\end{corollary}

\begin{proof}
By the resolution of the density conjecture \cite{MR2079598, MR2925381, MR3001608, MR2821565}, one has that $\Gamma$ is isomorphic to a geometrically finite subgroup of $\mathrm{Isom}(\mathbb{H}^3)$ which we continue to denote by $\Gamma$. By Brooks' theorem \cite{MR860677}, there is a discrete and faithful deformation of the inclusion $\Gamma \rightarrow \mathrm{Isom}(\mathbb{H}^3)$ whose image is contained in a lattice within $\mathrm{Isom}(\mathbb{H}^3)$, so that we may assume $\Gamma$ is itself a lattice in $\mathrm{Isom}(\mathbb{H}^3)$. One then has by work of Agol \cite{MR3104553} and Wise \cite[\S17.c]{MR4298722} that $\Gamma$ is virtually compact special. If $\Lambda$ is geometrically finite within $\mathrm{Isom}(\mathbb{H}^3)$, then $\Lambda$ is relatively quasiconvex in $\Gamma$ (with respect to the maximal parabolic subgroups of $\Gamma$; see \cite[Thm.~1.3]{MR2587434}), and hence the double $\Gamma \ast_\Lambda \Gamma$ is linear by Theorem~\ref{virtuallyspecialrelativelyhyperbolic}. Otherwise, by Canary's covering theorem \cite{MR1396777} and the resolution of the tameness conjecture \cite{arXiv:math/0405568, MR2188131}, there is a finite-index normal subgroup $\Gamma'$ of $\Gamma$ and a map $\varphi: \Gamma' \rightarrow \mathbb{Z}$ such that $\Lambda' := \Gamma' \cap \Lambda$ coincides with the kernel of~$\varphi$. Now choose some faithful representation $\rho'_0: \Gamma' \rightarrow \mathsf{O}(d)$, as guaranteed by Agol \cite{agol2018hyperbolic}, and define $\rho': \Gamma' \rightarrow \mathsf{GL}_{d+1}(\mathbb{R})$ by
\[
\rho(\gamma) = \begin{pmatrix} 
\gamma & \\
& e^{\varphi(\gamma)}
\end{pmatrix}
\]
for $\gamma \in \Gamma'$. Then $\rho(\Lambda')$ is the intersection of $\rho(\Gamma')$ with the compact subgroup
\[
\left\{ \begin{pmatrix} 
g & \\
& 1
\end{pmatrix}\right\}_{g \in \mathsf{O}(d)}
\]
of $\mathsf{GL}_{d+1}(\mathbb{R})$. It thus follows from Lemma~\ref{retractinduced} and Theorem~\ref{compact-1} that the double $\Gamma \ast_\Lambda \Gamma$ is again linear in this case.
\end{proof}

We now proceed to the proofs of Theorems~\ref{absolutelysimple} and~\ref{anisotropic}.

\begin{proof}[Proof~of~Theorem~\ref{absolutelysimple}]
Let $K \subset \mathbb{R}$ be the adjoint trace field of $\Gamma$, and denote by $\mathcal{O}_K$ the ring of integers of $K$. Since $\mathsf{G}$ is absolutely almost simple, if $\Gamma$ is an arithmetic subgroup of $G$, we have (see, for instance, \cite[Lem.~2.6]{MR2511587}) that $K$ is a totally real number field, and there exist a $K$-group ${\bf G} < \mathsf{SL}_d$, a local isomorphism $\phi: \mathsf{G} \rightarrow {\bf G}(\mathbb{R})$, and a finite-index normal subgroup $\Gamma' < \Gamma$ such that
\begin{itemize}
\item $\phi$ is injective on $\Gamma'$;
\item $\phi(\Gamma')$ is a finite-index subgroup of ${\bf G}(\mathcal{O}_K)$;
\item ${\bf G}^\sigma(\mathbb{R})$ is compact for each non-identity embedding $\sigma: K \rightarrow \mathbb{R}$.
\end{itemize}

Let $\Lambda' := \Gamma' \cap \Lambda$, and let $\varphi: \Gamma \rightarrow \mathsf{GL}_n(\mathbb{R})$ be the representation of $\Gamma$ induced by $\phi: \Gamma' \rightarrow \mathsf{SL}_d(\mathbb{R})$. Since $\Lambda'$ is of finite index in $\Lambda$ and the latter is cocompact in $\overline{\Lambda}^\mathrm{Zar}$, we have that $\Lambda'$ is cocompact in $\overline{\Lambda}^\mathrm{Zar}$ and hence cocompact in $\overline{\Lambda'}^\mathrm{Zar}$. It follows that $\Lambda'$ is of finite index in $\Gamma \cap \overline{\Lambda'}^\mathrm{Zar}$. Since the restriction of $\varphi$ to $\Gamma'$ extends to an $\mathbb{R}$-algebraic representation of $\mathsf{G}$, we have that $\varphi(\Lambda')$ is of finite index in $\varphi(\Gamma') \cap \overline{\varphi(\Lambda')}^\mathrm{Zar}$. Since the latter is of finite index in $\varphi(\Gamma) \cap \overline{\varphi(\Lambda)}^\mathrm{Zar}$, and since $\varphi(\Lambda)$ contains $\varphi(\Lambda')$, we conclude that $\varphi(\Lambda)$ is of finite index in $\varphi(\Gamma) \cap  \overline{\varphi(\Lambda)}^\mathrm{Zar}$.

Since we are assuming that $K \neq \mathbb{Q}$, there is at least one non-identity embedding $\sigma_0 : K \rightarrow \mathbb{R}$. We may then postconjugate $\varphi$ such that $\varphi(\Gamma) \subset \mathsf{GL}_n(K)$ and $\varphi(\Gamma')^{\sigma_0}$ is precompact in $\mathsf{GL}_n(\mathbb{R})$. The finite-index supergroup $\varphi(\Gamma)^{\sigma_0}$ of $\varphi(\Gamma')^{\sigma_0}$  is then also precompact in $\mathsf{GL}_n(\mathbb{R})$. We may now apply Corollary~\ref{compactcorollaries}(\ref{precompactgaloisconjugate}) to conclude that the double $\Gamma \ast_\Lambda \Gamma$ is linear. This proves Theorem~\ref{absolutelysimple}(\ref{absolutelysimple1}).

Now suppose $\mathsf{G}$ is not locally isomorphic to $\mathsf{O}(n,1)$ or $\mathsf{U}(n,1)$ for any $n \in \mathbb{N}$. Then $\Gamma$ is an arithmetic subgroup of $\mathsf{G}$ by the arithmeticity theorems of Margulis \cite{Mar} and Gromov--Schoen~\cite{GS}, and hence we conclude from Theorem~\ref{absolutelysimple}(\ref{absolutelysimple1}) that the double $\Gamma \ast_\Lambda \Gamma$ is linear as soon as $\Lambda$ is cocompact in $\overline{\Lambda}^\mathrm{Zar}$ and separable in $\Gamma$. We also necessarily have that $\Gamma$ is cocompact in $\mathsf{G}$; indeed, taking $\Gamma' < \Gamma$ and $\phi$ as in the proof of (\ref{absolutelysimple1}), since $\phi(\Gamma')^{\sigma_0}$ is precompact, we have that $\phi(\Gamma')^{\sigma_0}$ contains no nontrivial unipotent elements, and hence neither does $\Gamma'$, so that~$\Gamma'$ is cocompact in $\mathsf{G}$ by Godement's criterion (see, for instance, \cite[Prop.~5.3.1]{MR3307755}). Now if~$\Lambda$ is not separable in $\Gamma$, then the double $\Gamma \ast_\Lambda \Gamma$ is not residually finite, let alone linear. Finally, if $\Lambda$ is not cocompact in $\overline{\Lambda}^\mathrm{Zar}$, then since $\Gamma$ is cocompact in $\mathsf{G}$, we have that $\Lambda$ is of infinite covolume in $\overline{\Lambda}^\mathrm{Zar}$, and hence the double $\Gamma \ast_\Lambda \Gamma$ again fails to be linear by \cite[Cor.~5.3]{TT-IMRN}.
\end{proof}

\begin{remark}\label{indefiniteorthogonal}
\normalfont{
By Margulis's arithmeticity theorem \cite{Mar}, any lattice $\Gamma < \mathsf{O}(p,q)$ for $p \geq q \geq 2$ and $p+q \geq 5$ is arithmetic. If moreover $p+q$ is odd, then it follows from the classification of arithmetic subgroups of $\mathsf{O}(p,q)$ that $\Gamma$ is commensurable (in the wide sense) to $\mathsf{O}(f; \mathcal{O}_K)$, where $\mathcal{O}_K$ is the ring of integers of some totally real number field $K \subset \mathbb{R}$, and $f$ is a quadratic form of signature $(p,q)$ with coefficients in $K$; see \cite[pp.~380]{MR3307755}. In this case, the field $K$ coincides with the adjoint trace field of $\Gamma$ (again by \cite[Lem.~2.6]{MR2511587}, for instance). If $\Gamma$ is moreover cocompact, then $f$ must be anisotropic over $K$ and thus $K \neq \mathbb{Q}$ by Meyer's theorem. 
}
\end{remark}

\begin{proof}[Proof~of~Theorem~\ref{anisotropic}]
Let $\Gamma'$ be a finite-index normal subgroup of $\Gamma$ such that $\Gamma' \subset {\bf G}(\mathcal{O}_{K,S})$, and let $\Lambda' := \Gamma' \cap \Lambda$. Then the projection $\phi: \mathsf{G} \rightarrow {\bf G}(K_w)$ is faithful on ${\bf G}(\mathcal{O}_{K,S})$ and hence on~$\Gamma'$. Let $\varphi : \Gamma \rightarrow \mathsf{GL}_n(\mathbb{R})$ be the representation induced by~$\phi \bigr|_{\Gamma'}$. Then, since $\phi(H) = {\bf H}(K_w)$ is compact, we have as in the proof of Theorem~\ref{absolutelysimple} that the closure $\overline{\varphi(\Lambda)} \subset \mathsf{GL}_n(\mathbb{R})$ is compact and that $\varphi(\Lambda)$ is of finite index in $\varphi(\Gamma) \cap \overline{\varphi(\Lambda)}$. Since $\Lambda$ was moreover assumed separable in $\Gamma$, we conclude from Corollary~\ref{compactcorollaries}(\ref{compact-2}) that the double $\Gamma \ast_\Lambda \Gamma$ is linear over $\mathbb{R}$.
\end{proof}

\section{Doubling along cocompact stabilizers}\label{doublingalongreflective}

In this section, we prove Theorem~\ref{reflective}. We will need the following consequence of Theorem~\ref{main}.

\begin{proposition}
Let $\Gamma_1,\Gamma_2$ be groups with a common subgroup $\Lambda$. Suppose we are given for $j=1,2$ and for some integral domain $\mathsf{R}$ faithful representations $\phi^j: \Gamma_j \rightarrow \mathsf{GL}_n(\mathsf{R})$ and $\psi^j: \Gamma_j \rightarrow \mathsf{GL}_n(\mathsf{R})$ such that
\begin{itemize}
\item the representations $\phi^1, \phi^2, \psi^1$, and $\psi^2$ all coincide on $\Lambda$; and
\item for $j=1,2$ and $\gamma \in \Gamma_j$, if $\phi^j(\gamma)^{-1}\psi^j(\gamma)$ is unipotent, then $\gamma \in \Lambda$.
\end{itemize}
Then, for $i=1, \ldots n+1$ and $j=1, 2$, there are representations $\rho_i^j : \Gamma_j \rightarrow \mathsf{GL}_{n_i}(\mathsf{R})$, where $n_i = \binom{n+1}{i}$, all of which are faithful with the exception of the  $\rho_{n+1}^j$, satisfying the hypotheses of Lemma~\ref{multiple}. In particular, there is a faithful representation of $\Gamma_1 \ast_\Lambda \Gamma_2$ into $\mathsf{GL}_r(\mathsf{R}[s,t])$, where 
\[
r=\left(\binom{2n+2}{n+1}-1\right)^2+1.
\]   
\end{proposition}

\begin{proof} By composing the $\phi^j$ and the $\psi^j$ with the embedding $\mathsf{GL}_{n}(\mathsf{R})\times \{1\} \xhookrightarrow{} \mathsf{GL}_{n+1}(\mathsf{R})$, we obtain representations $\Gamma_j \rightarrow \mathsf{GL}_{n+1}(\mathsf{R})$, which we continue to denote by $\phi^j$ and $\psi^j$, fixing the standard basis vector $e_{n+1}$ and mapping no nontrivial matrix to a scalar matrix. Now for $i=1, \ldots n+1$ and $j=1, 2$, let $\rho_i^{j} : \Gamma_j \rightarrow \mathsf{GL}(\mathrm{End}(\wedge^{i}\mathsf{R}^{n+1})\big)$ be the representation given by
\[
\rho_i^{j}(\gamma)(M) = \big(\wedge^{i}\phi^j(\gamma)\big)M\big(\wedge^{i}\psi^j(\gamma)\big)^{-1}
\]
for $M\in \mathrm{End}(\wedge^{j}\mathsf{R}^{n+1})$ and $\gamma \in \Gamma_j$. Note that, for $i=1, \ldots, n$, the $\rho_i^{j}$ are faithful. Indeed, if $\rho_i^{j}(\gamma)=\textup{Id}$, then, by considering the action on the endomorphism $M=\mathrm{Id}$, we see that $\wedge^{i}\phi^j(\gamma) = \wedge^{i}\psi^j(\gamma)$, and hence $\wedge^{i}\phi^j(\gamma)$ commutes with all endomorphisms of $\wedge^{j}\mathsf{R}^{n+1}$, so that $\wedge^{i}\phi^j(\gamma)$ is scalar. For $i < n+1$, this implies that $\phi^j(\gamma)$ is itself scalar, so that $\gamma = \mathrm{Id}$.

Since $\phi^j$ and $\psi^j$ coincide on $\Lambda$, we have that $\rho_i^{j}(\Lambda)$ fixes the identity map $\mathrm{Id} \in \mathrm{End}(\wedge^{i}\mathsf{R}^{n+1})$ and preserves the $\mathsf{R}$-submodule $V_i:=\{M\in \textup{End}(\wedge^{i}\mathsf{R}^{n+1}):\textup{tr}(M)=0\}$. In addition, for every $\gamma \in \Gamma_j$, we have 
\[
\rho_i^{j}(\gamma)(\mathrm{Id})=\wedge^{i}\big(\phi^j(\gamma)\psi^j(\gamma)^{-1}\big) =\textup{tr}\big(\wedge^{i}\big(\phi^j(\gamma)\psi^j(\gamma)^{-1}\big)\big)\mathrm{Id}+M_{ij\gamma}
\]
for some $M_{ij\gamma}\in V_i$. Thus, if $\mathcal{B}_i$ is an ordered basis for $\textup{End}(\wedge^{i}\mathsf{R}^{n+1})$ whose first basis vector is the identity endomorphism $\mathrm{Id}$ and whose remaining basis vectors are contained in $V_i$, then, writing $\rho_i^{j}(\gamma)$ with respect to the basis $\mathcal{B}_i$, we have that the top-left entry of $\rho_i^{j}(\gamma)$ is
\begin{equation}\label{char-poly}
\left( \rho_i^{j}(\gamma)\right)_{11} = \mathrm{tr}\big(\wedge^j(\phi^j(\gamma)\psi^j(\gamma)^{-1})\big)
\end{equation}
for any $\gamma \in \Gamma_j$.

We now check that the hypotheses of Lemma~\ref{multiple} are satisfied for the representations $\rho_i^j$. 
We have already observed that $\rho_i^j(\Lambda)$ fixes the identity endomorphism of $\wedge^{i}\mathsf{R}^{n+1}$ and preserves $V_i$, so that $\rho_i^j(\Lambda) \subset \{1\} \times \mathsf{GL}_{n_i-1}(\mathsf{R})$. That, for $i=1, \ldots, n+1$, the representations $\rho_i^1$ and $\rho_i^2$ coincide on $\Lambda$ follows immediately from the fact that the representations $\phi^1, \phi^2, \psi^1$, and $\psi^2$ all coincide on $\Lambda$. Finally, if $\gamma \in \Gamma_j$ and $\left( \rho_i^{j}(\gamma)\right)_{11} = 1$ for $i=1,\ldots, n+1$, then, by (\ref{char-poly}), the characteristic polynomial of $\phi^j(\gamma)\psi^j(\gamma)^{-1}$ coincides with that of the identity matrix, i.e., the matrix $\phi^j(\gamma)\psi^j(\gamma)^{-1}$ is unipotent, and hence $\gamma \in \Lambda$ by assumption. Thus, by applying Lemma~\ref{multiple} to the representations $\rho_i^j$, we obtain a faithful representation of $\Gamma_1 \ast_\Lambda \Gamma_2$ into $\mathsf{GL}_{r}(\mathsf{R}[s,t])$ where $r=d^2+1$ and 
\[
d=\sum_{i=1}^{n+1}\dim_{\mathsf{R}}\big(\textup{End}(\wedge^{i}\mathsf{R}^{n+1})\big)=\sum_{i=1}^{n+1}\binom{n+1}{i}^2=\binom{2n+2}{n+1}-1.\qedhere\]
\end{proof}

One proves in precisely the same manner the following more technical statement.

\begin{theorem}\label{unipotent-free-1} Let $\Gamma_1,\Gamma_2$ be groups with a common subgroup $\Lambda$. Suppose we are given for $j=1,2$, for some integer $\ell \geq 1$, for $k=1, \ldots, \ell$, and for some integral domain $\mathsf{R}$ faithful representations $\phi_k^j: \Gamma_j \rightarrow \mathsf{GL}_n(\mathsf{R})$ and $\psi_k^j: \Gamma_j \rightarrow \mathsf{GL}_n(\mathsf{R})$ such that
\begin{itemize}
\item the representations $\phi_k^j$ and $\psi_{k'}^{j'}$ all coincide on $\Lambda$; and
\item for $j=1,2$ and $\gamma \in \Gamma_j$, if $\phi_k^j(\gamma)^{-1}\psi_k^j(\gamma)$ is unipotent for every $k=1, \ldots, \ell$, then $\gamma \in \Lambda$.
\end{itemize}
Then, for $i=1, \ldots n+1$, for $j=1, 2$, and for $k=1, \ldots, \ell$, there are representations $\rho_i^{j,k} : \Gamma_j \rightarrow \mathsf{GL}_{n_i}(\mathsf{R})$, where $n_i = \binom{n+1}{i}$, all of which are faithful with the exception of the $\rho_{n+1}^{j,k}$, satisfying
\begin{itemize}
\item $\rho_i^{j,k}(\Lambda) \subset \{1\} \times \mathsf{GL}_{n_i-1}(\mathsf{R})$;
\item the representations $\rho_i^{1,k}$ and $\rho_i^{2,k}$ coincide on $\Lambda$ for $i=1, \ldots, n+1$ and $k=1, \ldots, \ell$; and
\item for $j=1,2$ and $\gamma \in \Gamma_j$, if the top-left entry of $\rho_i^{j,k}(\gamma)$ is equal to $1$ for every $i$ and $k$, then $\gamma \in \Lambda$.
\end{itemize}
In particular, there is a faithful representation of $\Gamma_1 \ast_\Lambda \Gamma_2$ into $\mathsf{GL}_r(\mathsf{R}[s,t])$, where 
\[
r=\ell^2\left(\binom{2n+2}{n+1}-1\right)^2+1.
\]\end{theorem}

\medskip

We record the following consequence of Theorem~\ref{unipotent-free-1}.

\begin{theorem}\label{unipotent-free} Let $\mathsf{R}$ be an integral domain, let $\Gamma <\mathsf{GL}_m(\mathsf{R})$ be a subgroup, and suppose one has representations $\varphi_0, \ldots,  \varphi_\ell:\Gamma \rightarrow \mathsf{GL}_m(\mathsf{R})$ the subgroup generated by whose images contains no nontrivial unipotent elements. Then the double $\Gamma \ast_\Lambda \Gamma$ is linear, where 
\[
\Lambda :=\big\{\gamma \in \Gamma :\varphi_0(\gamma)=\cdots=\varphi_\ell(\gamma)\big\}.
\] \end{theorem}

\begin{proof}[Proof of Theorem \ref{unipotent-free}] For $\gamma \in \Gamma$, we have by our assumptions that $\gamma \in \Lambda$ if and only if  $\varphi_0(\gamma)^{-1}\varphi_k(\gamma)$ is unipotent for every $k= 1, \ldots, \ell$. We may now apply Theorem~\ref{unipotent-free-1}
with $\Gamma_1 = \Gamma_2 = \Gamma$, and with $\phi_k^1 = \phi_k^2 = \mathrm{id}_\Gamma \times \varphi_0$ and $\psi_k^1 = \psi_k^2 = \mathrm{id}_\Gamma \times \varphi_k$ for $k=1, \ldots, \ell$, where $\mathrm{id}_\Gamma \times \varphi_k : \Gamma \rightarrow \mathsf{GL}_{2m}(\mathsf{R})$ is the representation given by $\gamma \mapsto \begin{pmatrix}[0.8] \gamma &  \\  & \varphi_k(\gamma) \end{pmatrix}$ for $k=0, \ldots, \ell$. We obtain in this manner a faithful representation of the double $\Gamma \ast_\Lambda \Gamma$ into $\mathsf{GL}_r(\mathsf{R}[s,t])$, where
where $r=\ell^2\big(\binom{4m+2}{2m+1}-1\big)^2+1$.\end{proof}

\begin{Exmp}
\normalfont{
Let $\mathsf{G}$ be any semisimple real algebraic group, and fix a Langlands decomposition~$\mathsf{MAN}$ of a minimal parabolic subgroup of $\mathsf{G}$. Then for any cocompact lattice $\Gamma < \mathsf{G}$, if $\mathsf{H}$ is a Cartan subgroup of $\mathsf{G}$, i.e., a conjugate of $\mathsf{MA}$ in $\mathsf{G}$, such that $\Lambda:= \Gamma \cap \mathsf{H}$ is cocompact in~$\mathsf{H}$, then the double $\Gamma \ast_\Lambda \Gamma$ is linear. (That such $\mathsf{H}$ always abound follows from Poincar\'e recurrence on the compact homogeneous space~$\Gamma \backslash \mathsf{G}$, together with a well-known argument of Selberg \cite[Lem.~1.10]{MR302822}.) Indeed, since $\Lambda$ is cocompact in $\mathsf{H}$, there is some element $h \in \Lambda$ such that $\mathsf{H}$ is precisely the centralizer of $h$ in $\mathsf{G}$. Linearity of the double $\Gamma \ast_\Lambda \Gamma$ thus follows from Theorem~\ref{unipotent-free} if one takes $\varphi_0$ to be the inclusion $\Gamma \rightarrow \mathsf{G}$ and $\varphi_1$ to be the inclusion postconjugated by $h$. Note that Theorem~\ref{compact-1} cannot be applied to deduce linearity of some of these doubles; for example, if $\mathsf{G}$ is absolutely almost simple and of real rank $\geq 2$ and $\Gamma$ has adjoint trace field $\mathbb{Q}$, then it follows from Margulis superrigidity~\cite{Mar} that the image of $\Lambda$ under no faithful finite-dimensional real representation of $\Gamma$ is precompact. Cocompact such $\Gamma$ exist for instance within $\mathsf{G} = \mathsf{SL}_d(\mathbb{R})$ for any $d \geq 3$ (see \cite{MR2655311}); note that Theorem~\ref{reflective} also does not apply to the latter examples since in those cases the pair $(\mathsf{G}, \mathsf{MA})$ fails to be affine symmetric.
}
\end{Exmp}

Theorem~\ref{reflective} now follows immediately from the following more general statement.

\begin{theorem}\label{reflective-2}
Let $X$ be a Riemannian symmetric space of noncompact type and $Y_1, \ldots, Y_\ell$ be reflective submanifolds of $X$ with nonempty intersection $Y$. Let $\Gamma < \mathrm{Isom}(X)$ be a discrete subgroup, let $\Lambda_k < \Gamma$ be the stabilizer of $Y_k$ in $\Gamma$ for $k=1, \ldots, \ell$, and let $\Lambda = \bigcap_{k=1}^\ell \Lambda_k$. If $\Lambda_k$ acts cocompactly on $Y_k$ for every $k$, then the double $\Gamma \ast_\Lambda \Gamma$ embeds in $\mathsf{GL}_r(\mathbb{R})$ for some dimension $r=r(X, \ell)$ depending only on $X$ and~$\ell$. 
\end{theorem}

\begin{proof}
Denote by $\mathsf{G} = \mathrm{Isom}(X)$, by $\mathfrak{g}$ the Lie algebra of $\mathsf{G}$, and by $\mathrm{Ad}: \mathsf{G} \rightarrow \mathsf{GL}(\mathfrak{g})$ the adjoint representation of $\mathsf{G}$. Note that $\mathrm{Ad}$ is faithful since the centralizer in $\mathsf{G}$ of the identity component of $\mathsf{G}$ is trivial. For $k=1, \ldots, \ell$, denote by $\sigma_k \in \mathsf{G}$ the isometric involution of $X$ whose fixed point set is $Y_k$. 

Now fix $k \in \{1, \ldots, \ell\}$ and $\gamma \in \Gamma$. We claim that $\gamma \in \Lambda_k$ if $\mathrm{Ad}(\gamma^{-1} \sigma_k \gamma \sigma_k)$ is unipotent. Indeed, suppose $\gamma \notin \Lambda_k$. In the case that $Y_k \cap \gamma^{-1} Y_k \neq \emptyset$, we have that $\gamma^{-1} \sigma_k \gamma \sigma_k$ is a nontrivial isometry of $X$ fixing $Y_k \cap \gamma^{-1} Y_k$ pointwise. In the case that $Y_k \cap \gamma^{-1} Y_k = \emptyset$, we have that the distance between $Y_k$ and $\gamma^{-1} Y_k$ is bounded below by the length $\delta > 0$ of a shortest (positive-length) orthogeodesic to $\Lambda_k \backslash Y_k$ in $\Gamma \backslash X$. Since $\sigma_k$ inverts every geodesic line in $X$ orthogonal to~$Y_k$, we have in this case that the isometry $\gamma^{-1}\sigma_k\gamma\sigma_k$ preserves each geodesic line $L$ in~$X$ orthogonal to both $Y_k$ and $\gamma^{-1} Y_k$ and translates by a distance of $\geq \delta$ along~$L$. Thus, in either case, we have that $\gamma^{-1} \sigma_k \gamma \sigma_k$ is a nontrivial semisimple isometry of~$X$, and hence that $\mathrm{Ad}(\gamma^{-1} \sigma_k \gamma \sigma_k)$ is a nontrivial semisimple matrix, so that $\mathrm{Ad}(\gamma^{-1} \sigma_k \gamma \sigma_k)$ in particular fails to be unipotent. 

To conclude, we may now apply Theorem~\ref{unipotent-free-1} with $\Gamma_1 = \Gamma_2 = \Gamma$, with $\phi_k^1 = \phi_k^2 = \mathrm{Ad}$, with $\psi_k^1 = \psi_k^2$ the representation given by precomposing $\mathrm{Ad}$ with conjugation by $\sigma_k$, and with $\mathsf{R}$ the entry field of the subgroup of $\mathsf{GL}(\mathfrak{g})$ generated by $\mathrm{Ad}(\Gamma), \mathrm{Ad}(\sigma_1 \Gamma \sigma_1), \ldots, \mathrm{Ad}(\sigma_\ell\Gamma\sigma_\ell)$. We may in particular take
\[
r(X, \ell) = \ell^2\left(\binom{2\dim_{\mathbb{R}}\mathfrak{g}+2}{\dim_{\mathbb{R}}\mathfrak{g}+1}-1\right)^2+1. \qedhere
\] \end{proof}

\begin{Exmp}
\normalfont{
Let $\mathsf{G} = \mathsf{SL}_2(\mathbb{R}) \times \mathsf{SL}_2(\mathbb{R})$, and let $\Gamma < \mathsf{G}$ be any irreducible noncocompact lattice. Then it follows from Margulis's arithmeticity theorem, together with the fact that any quadratic field embeds in some quaternion division algebra over $\mathbb{Q}$, that there are conjugates~$\mathsf{H}$ of the diagonal subgroup $\{(g,g)\}_{g \in \mathsf{SL}_2(\mathbb{R})} < \mathsf{G}$ such that $\Lambda:= \Gamma \cap \mathsf{H}$ is cocompact in $\mathsf{H}$. One then deduces from Theorem~\ref{reflective} that any such double $\Gamma \ast_\Lambda \Gamma$ is linear. These are examples to which Theorem~\ref{compact-1} does not apply, since it again follows from Margulis superrigidity that in these cases the image of $\Lambda$ under no faithful finite-dimensional real representation of $\Gamma$ is precompact.
}
\end{Exmp}

\begin{Exmp}
\normalfont{
Let $M$ be an arithmetic $\mathbb{K}$-hyperbolic manifold of simplest type and of $\mathbb{K}$-dimension $d \geq 2$, where $\mathbb{K} = \mathbb{R},$ $\mathbb{C}$, or $\mathbb{H}$. Suppose also that the adjoint trace field of (the fundamental group of) $M$ is quadratic, so that $M$ is in particular compact. Then, as argued in~\cite[Cor.~4.5]{arXiv:2105.06897} in the case $\mathbb{K} = \mathbb{R}$, the manifold~$M$ is commensurable via restriction of scalars to an immersed totally geodesic $\mathbb{K}$-submanifold $\Lambda \backslash Y$ of a noncompact arithmetic $\mathbb{K}$-hyperbolic manifold $\Gamma \backslash X$, where~$X$ is of $\mathbb{K}$-dimension $2d+1$. It then follows from Theorem~\ref{reflective} that the double $\Gamma \ast_\Lambda \Gamma$ is linear. For $\mathbb{K} = \mathbb{H}$, this again gives examples to which Theorem~\ref{compact-1} does not apply, this time by Corlette superrigidity \cite{Corlette}.
}
\end{Exmp}

\begin{Exmp}
\normalfont{
Let $f_1$ and $f_2$ be the $\mathbb{Q}$-anisotropic rational quadratic forms $x_1^2+x_2^2-7x_3^2$ and $17x_1^2+x_2^2-7x_3^2$, respectively. (The $f_i$ are indeed $\mathbb{Q}_2$-anisotropic again since $-7$ and $17$ are both squares in $\mathbb{Q}_2$ and the stufe of $\mathbb{Q}_2$ is $4$.) Then $(\mathsf{SL}_3(\mathbb{R}), \mathsf{SO}(f_i; \mathbb{R}))$ is an affine symmetric pair and $\mathsf{SO}(f_i; \mathbb{Z})$ is cocompact in $\mathsf{SO}(f_i; \mathbb{R})$ for $i=1,2$. It then follows from Theorem~\ref{reflective-2} that the double $\Gamma \ast_\Lambda \Gamma$ is linear, where $\Gamma := \mathsf{SL}_3(\mathbb{Z})$, and where $\Lambda$ is the virtually infinite cyclic subgroup
\[
\Lambda := \mathsf{SO}(f_1; \mathbb{Z}) \cap \mathsf{SO}(f_2; \mathbb{Z}) = \mathsf{S}(\{\pm 1\} \times \mathsf{O}(f; \mathbb{Z}));
\]
here $f$ is the quadratic form in two variables given by $x^2-7y^2$. We include this example because the pair $(\mathsf{SL}_3(\mathbb{R}),  \mathsf{S}(\{\pm 1\} \times \mathsf{O}(f; \mathbb{R}))$ is {\em not} affine symmetric.
}
\end{Exmp}

\section{Doubling along noncocompact lattices} \label{nonlinear-amalgams}

In this section, we use superrigidity to prove Theorem \ref{negative}.

\begin{proof}[Proof~of~Theorem~\ref{negative}]
If $\Lambda$ is not of finite covolume in $\overline{\Lambda}^{\mathrm{Zar}}$, then $\Gamma \ast_\Lambda \Gamma$ is not linear by \cite[Cor.~5.3]{TT-IMRN}. We may thus assume that $\Lambda$ is a noncocompact lattice in $\overline{\Lambda}^{\mathrm{Zar}}$, so that $\Gamma$ is also not cocompact in~$\mathsf{G}$. By Margulis's arithmeticity theorem \cite{Mar} (see also \cite[Cor.~5.3.2]{MR3307755}), there is then an algebraically connected and simply connected $\mathbb{Q}$-group ${\bf G}$ and a morphism $\phi: {\bf G}(\mathbb{R}) \rightarrow \mathsf{G}$ with finite kernel such that $\Gamma$ is commensurable with $\phi({\bf G}(\mathbb{Z}))$. Let $\Gamma'$ be a finite-index subgroup of ${\bf G}(\mathbb{Z})$ such that $\phi$ is injective on $\Gamma'$ and $\phi(\Gamma') \subset \Gamma$, and let $\Lambda' = \phi^{-1}(\Lambda) \cap \Gamma'$. Then it suffices to show that $\Gamma' \ast_{\Lambda'} \Gamma'$ is not linear. 

By our assumption on $\Lambda$, we have that $\Lambda'$ is of finite index in ${\bf H}(\mathbb{Z})$ for some isotropic $\mathbb{Q}$-subgroup ${\bf H} < {\bf G}$ of positive codimension. Let $g$ be a nontrivial element of ${\bf T}(\mathbb{R})^\circ$ for some positive-dimensional $\mathbb{Q}$-split torus ${\bf T} < {\bf H}$, and denote by $U(g^{\pm 1}) < {\bf G}(\mathbb{R})$ the (stable) horospherical subgroup of $g^{\pm 1}$. Since ${\bf G}(\mathbb{R})^\circ$ is generated by $U(g)$ and $U(g^{-1})$ (see \cite[Thm.~I.2.3.1]{Mar}), we must have that ${\bf H}(\mathbb{R}) \cap U(g^\epsilon)$ is of positive codimension in $U(g^\epsilon)$ for some $\epsilon \in \{\pm 1\}$; assume henceforth without loss of generality that $\epsilon = 1$. On the other hand, since $U(g)$ is defined over~$\mathbb{Q}$, we have that $\Gamma' \cap U(g)$ is a lattice in $U(g)$. Thus, there is some $u \in \Gamma' \cap U(g)$ such that $\langle u \rangle \cap {\bf H}(\mathbb{R}) = \{1\}$.

Now suppose one has a faithful representation $\rho: \Gamma' \ast_{\Lambda'} \Gamma' \rightarrow \mathsf{GL}_d(\mathbb{F})$ for some field $\mathbb{F}$ and $d \in \mathbb{N}$. Since $\Gamma'$ is itself not linear over any field of positive characteristic by our assumption on~$\mathsf{G}$, we must have that $\mathrm{char}(\mathbb{F}) = 0$.  Since $\Gamma'$ is finitely generated, and since $\mathsf{GL}_d(\mathbb{C})$ embeds into $\mathsf{GL}_{2d}(\mathbb{R})$, we may assume $\mathbb{F}=\mathbb{R}$. By Margulis superrigidity \cite{Mar} in the case $\textup{rk}_{\mathbb{R}}(\mathsf{G})\geq 2$, and by  Corlette \cite{Corlette} and Gromov--Schoen \cite{GS} superrigidity in the case  $\textup{rk}_{\mathbb{R}}(\mathsf{G})=1$, there exist finite-index subgroups $\Gamma_i<\Gamma'$ and continuous representations \hbox{$\overline{\rho}_i: {\bf G}(\mathbb{R}) \rightarrow \mathsf{GL}_d(\mathbb{R})$ such that} \begin{align*} \overline{\rho}_i(\gamma)&=\rho(\gamma),\ \gamma\in \Gamma_i,\\ \overline{\rho}_1(h)&=\overline{\rho}_2(h), \ h\in \Lambda' \cap \Gamma_1\cap  \Gamma_2.\end{align*} Since ${\bf G}$ is algebraically simply connected, by \cite[Thm. 3.3.4]{On-Vin} the representations $\rho_1,\rho_2$ are $\mathbb{R}$-algebraic. Thus, there is a finite-index subgroup $\mathsf{H}< {\bf H}(\mathbb{R})$ containing $\Lambda' \cap\Gamma_1\cap \Gamma_2$ with $\overline{\rho}_1(h)=\overline{\rho}_2(h)$ for every $h\in \mathsf{H}$. In particular, we have $\overline{\rho}_1(g)=\overline{\rho}_2(g)$. Denoting by $\mathcal{U}(g)$ the horospherical subgroup of $\overline{\rho}_1(g)=\overline{\rho}_2(g)$ in $\mathsf{GL}_d(\mathbb{R})$, we then have that $\overline{\rho}_i(u^m)\in \mathcal{U}(g)$ for $i=1,2$ and for $m>0$ such that $u^m \in \Lambda' \cap \Gamma_1 \cap \Gamma_2$. Now since $\rho$ was assumed faithful, we have that $\langle \overline{\rho}_1(u^m),\overline{\rho}_2(u^m) \rangle < \mathcal{U}(g)$ decomposes as the free product $\langle \overline{\rho}_1(u^m) \rangle \ast \langle \overline{\rho}_2(u^m) \rangle$. But this is absurd since $\mathcal{U}(g)$ is a nilpotent group.
\end{proof}

\begin{remark}
\normalfont{Note that if $\Lambda$ is any finite-index subgroup of a finitely generated linear group $\Gamma$, then, since amalgams of finite groups are virtually free and hence linear, the amalgam $\Gamma\ast_\Lambda \Gamma$ is also linear.}
\end{remark}

\begin{remark}
\normalfont{What follows is a straightforward demonstration of the failure of Theorem~\ref{negative} for certain pairs $(\Gamma, \Lambda)$, where $\Gamma$ is a lattice in $\mathsf{O}(n,1)$, $3 \leq n \leq 8$, and $\Lambda = \Gamma \cap \mathsf{H}$ is of finite covolume in some conjugate $\mathsf{H}$ of $\mathsf{O}(n-1,1)$ in $\mathsf{O}(n,1)$. Indeed, for any such $n$, there is a noncompact finite-volume hyperbolic polyhedron $Q$ of dimension $n$; see \cite{MR2130566}. Let $P$ be the polyhedron obtained from~$Q$ by doubling $P$ along some noncompact codimension-$1$ face $F$ of~$Q$. Let $\Gamma < \mathsf{O}(n,1)$ be the group generated by the reflections in the walls of $P$, and let $\Lambda < \Gamma$ be the group generated by the reflections in those walls of $P$ that are orthogonal to $F$. Then the pair $(\Gamma, \Lambda)$ is as above, but the double $\Gamma \ast_\Lambda \Gamma$ abstractly remains a (right-angled) Coxeter group and is thus linear by classical work of Tits \cite{MR240238} and Vinberg \cite{MR302779}. 
}
\end{remark}

\begin{corollary}\label{negative-examples}
Let $\mathsf{G}$ and $\Gamma < \mathsf{G}$ be as in Theorem \ref{negative}. If $\Gamma$ is not cocompact in $\mathsf{G}$, then $\Gamma$ contains a subgroup $\Lambda$ such that the double $\Gamma \ast_\Lambda \Gamma$ is residually finite but not linear.  
\end{corollary}

\begin{proof}
As in the proof of Theorem~\ref{negative}, it follows from Margulis's arithmeticity theorem that there is a $\mathbb{Q}$-group ${\bf G}$ and a local isomorphism $\phi: {\bf G}(\mathbb{R}) \rightarrow \mathsf{G}$ such that $\Gamma$ is commensurable with $\phi({\bf G}(\mathbb{Z}))$. Since $\Gamma$ is not cocompact in $\mathsf{G}$, Godement's criterion (see, for instance, \cite[Prop.~5.3.1]{MR3307755}) guarantees the existence of a unipotent element $u \in {\bf G}(\mathbb{Z})$. It now follows from the Jacobson--Morosov lemma (over $\mathbb{Q}$; see \cite[Ch.~3, Thm.~17]{MR559927}) that there is a $\mathbb{Q}$-subgroup ${\bf H} < {\bf G}$ such that ${\bf H}(\mathbb{R})$ is locally isomorphic to $\mathsf{SL}_2(\mathbb{R})$ and $u \in {\bf H}(\mathbb{Z})$. Let $\mathsf{H}:=\overline{\phi({\bf H}(\mathbb{R}))}^\mathrm{Zar} < \mathsf{G}$, and $\Lambda := \Gamma \cap \mathsf{H}$. Then~$\Lambda$ is separable in $\Gamma$ by \cite[Lemme~principal]{MR1769939}, so that the double $\Gamma \ast_\Lambda \Gamma$ is residually finite \cite[Prop.~1]{zbMATH03182231}. On the other hand, since $\mathsf{H}$ remains locally isomorphic to $\mathsf{SL}_2(\mathbb{R})$ and $\Lambda$ contains the unipotent element $\phi(u)$, we have that $\Lambda$ is not cocompact in $\Gamma$, so that the double $\Gamma \ast_\Lambda \Gamma$ is not linear by Theorem~\ref{negative}.
\end{proof}

\begin{remark}
\normalfont{In the case that the real rank of $\mathsf{G}$ is $\geq 2$, Corollary~\ref{negative-examples} is implicit in the work of Tholozan--Tsouvalas \cite{TT-IMRN}. Indeed, as observed in \cite{TT-IMRN}, it follows from work of Prasad--Rapinchuk~\cite{MR1960120} that in this case $\Gamma$ contains semisimple elements $\gamma$ such that $\overline{\langle \gamma \rangle}^\mathrm{Zar}$
} is of dimension $\geq 2$. One then concludes from \cite[Cor.~5.3]{TT-IMRN} that the double $\Gamma \ast_{\langle \gamma \rangle} \Gamma$ is not linear. However, any such double is residually finite, essentially because arithmetic subgroups of tori satisfy the congruence subgroup property \cite{MR44570}; see \cite[\S3]{MR2100678}.
\end{remark}

\bibliographystyle{siam}
\bibliography{biblio.bib}

@article {MR2115010,
    AUTHOR = {Drutu, Cornelia and Sapir, Mark},
     TITLE = {Non-linear residually finite groups},
   JOURNAL = {J. Algebra},
  FJOURNAL = {Journal of Algebra},
    VOLUME = {284},
      YEAR = {2005},
    NUMBER = {1},
     PAGES = {174--178},
      ISSN = {0021-8693,1090-266X},
   MRCLASS = {20E26},
  MRNUMBER = {2115010},
MRREVIEWER = {Avinoam\ Mann},
       DOI = {10.1016/j.jalgebra.2004.06.025},
       URL = {https://doi.org/10.1016/j.jalgebra.2004.06.025},
}

@Article{Corlette,
 Author = {Corlette, Kevin},
 Title = {Archimedean superrigidity and hyperbolic geometry},
 FJournal = {Annals of Mathematics. Second Series},
 Journal = {Ann. Math. (2)},
 ISSN = {0003-486X},
 Volume = {135},
 Number = {1},
 Pages = {165--182},
 Year = {1992},
 Language = {English},
 DOI = {10.2307/2946567},
 zbMATH = {55728},
 Zbl = {0768.53025}
}

@article {MR2029029,
    AUTHOR = {Erschler, Anna},
     TITLE = {Not residually finite groups of intermediate growth,
              commensurability and non-geometricity},
   JOURNAL = {J. Algebra},
  FJOURNAL = {Journal of Algebra},
    VOLUME = {272},
      YEAR = {2004},
    NUMBER = {1},
     PAGES = {154--172},
      ISSN = {0021-8693,1090-266X},
   MRCLASS = {20F05 (20F50 20F65)},
  MRNUMBER = {2029029},
MRREVIEWER = {Stephen\ P.\ Humphries},
       DOI = {10.1016/j.jalgebra.2002.11.005},
       URL = {https://doi.org/10.1016/j.jalgebra.2002.11.005},
}

@article {MR286898,
    AUTHOR = {Tits, J.},
     TITLE = {Free subgroups in linear groups},
   JOURNAL = {J. Algebra},
  FJOURNAL = {Journal of Algebra},
    VOLUME = {20},
      YEAR = {1972},
     PAGES = {250--270},
      ISSN = {0021-8693},
   MRCLASS = {20.75},
  MRNUMBER = {286898},
MRREVIEWER = {B.\ A. F. Wehrfritz},
       DOI = {10.1016/0021-8693(72)90058-0},
       URL = {https://doi.org/10.1016/0021-8693(72)90058-0},
}

@Article{Shalen,
 Author = {Shalen, Peter B.},
 Title = {Linear representations of certain amalgamated products},
 FJournal = {Journal of Pure and Applied Algebra},
 Journal = {J. Pure Appl. Algebra},
 ISSN = {0022-4049},
 Volume = {15},
 Pages = {187--197},
 Year = {1979},
 Language = {English},
 DOI = {10.1016/0022-4049(79)90033-1},
 zbMATH = {3622001},
 Zbl = {0401.20024}
}

@article {MR2221137,
    AUTHOR = {Labourie, Fran\c cois},
     TITLE = {Anosov flows, surface groups and curves in projective space},
   JOURNAL = {Invent. Math.},
  FJOURNAL = {Inventiones Mathematicae},
    VOLUME = {165},
      YEAR = {2006},
    NUMBER = {1},
     PAGES = {51--114},
      ISSN = {0020-9910,1432-1297},
   MRCLASS = {20F65 (37D20 37F30)},
  MRNUMBER = {2221137},
MRREVIEWER = {Richard\ Kenyon},
       DOI = {10.1007/s00222-005-0487-3},
       URL = {https://doi.org/10.1007/s00222-005-0487-3},
}

@book {MR607504,
    AUTHOR = {Serre, Jean-Pierre},
     TITLE = {Trees},
      NOTE = {Translated from the French by John Stillwell},
 PUBLISHER = {Springer-Verlag, Berlin-New York},
      YEAR = {1980},
     PAGES = {ix+142},
      ISBN = {3-540-10103-9},
   MRCLASS = {20H10 (05C05 22E50)},
  MRNUMBER = {607504},
}

@article{BF,
  title={A combination theorem for negatively curved groups},
  author={Bestvina, Mladen and Feighn, Mark},
  journal={Journal of Differential Geometry},
  volume={35},
  number={1},
  pages={85--101},
  year={1992},
  publisher={Lehigh University}
}

@article {MR2981818,
    AUTHOR = {Guichard, Olivier and Wienhard, Anna},
     TITLE = {Anosov representations: domains of discontinuity and
              applications},
   JOURNAL = {Invent. Math.},
  FJOURNAL = {Inventiones Mathematicae},
    VOLUME = {190},
      YEAR = {2012},
    NUMBER = {2},
     PAGES = {357--438},
      ISSN = {0020-9910,1432-1297},
   MRCLASS = {22F30 (32G15 53C30 53D25)},
  MRNUMBER = {2981818},
MRREVIEWER = {Pablo\ Su\'arez-Serrato},
       DOI = {10.1007/s00222-012-0382-7},
       URL = {https://doi.org/10.1007/s00222-012-0382-7},
}

@misc{arXiv:2510.21334,
 author = {Canary, Richard D.},
 title = {Kleinian viewpoints on higher rank worlds},
 year = {2025},
 howpublished = {Preprint, {arXiv}:2510.21334 [math.{GT}] (2025)},
 url = {https://arxiv.org/abs/2510.21334},
 arXiv = {arXiv:2510.21334}
}

@misc{arXiv:2504.21802,
 author = {Dey, Subhadip and Tsouvalas, Konstantinos},
 title = {Anosov representations of amalgams},
 year = {2025},
 howpublished = {Preprint, {arXiv}:2504.21802 [math.{GR}] (2025)},
 url = {https://arxiv.org/abs/2504.21802},
 arXiv = {arXiv:2504.21802}
}

@article {MR141672,
    AUTHOR = {Mostow, G. D. and Tamagawa, T.},
     TITLE = {On the compactness of arithmetically defined homogeneous
              spaces},
   JOURNAL = {Ann. of Math. (2)},
  FJOURNAL = {Annals of Mathematics. Second Series},
    VOLUME = {76},
      YEAR = {1962},
     PAGES = {446--463},
      ISSN = {0003-486X},
   MRCLASS = {22.55 (20.65)},
  MRNUMBER = {141672},
MRREVIEWER = {T.\ Ono},
       DOI = {10.2307/1970368},
       URL = {https://doi.org/10.2307/1970368},
}

@article {MR147566,
    AUTHOR = {Borel, Armand and Harish-Chandra},
     TITLE = {Arithmetic subgroups of algebraic groups},
   JOURNAL = {Ann. of Math. (2)},
  FJOURNAL = {Annals of Mathematics. Second Series},
    VOLUME = {75},
      YEAR = {1962},
     PAGES = {485--535},
      ISSN = {0003-486X},
   MRCLASS = {20.65 (14.50)},
  MRNUMBER = {147566},
MRREVIEWER = {P.\ Cartier},
       DOI = {10.2307/1970210},
       URL = {https://doi.org/10.2307/1970210},
}

@article {MR4012341,
    AUTHOR = {Bochi, Jairo and Potrie, Rafael and Sambarino, Andr\'es},
     TITLE = {Anosov representations and dominated splittings},
   JOURNAL = {J. Eur. Math. Soc. (JEMS)},
  FJOURNAL = {Journal of the European Mathematical Society (JEMS)},
    VOLUME = {21},
      YEAR = {2019},
    NUMBER = {11},
     PAGES = {3343--3414},
      ISSN = {1435-9855,1435-9863},
   MRCLASS = {22E40 (20F67 37B99 37D30 53C35)},
  MRNUMBER = {4012341},
MRREVIEWER = {Alejandro\ Ucan-Puc},
       DOI = {10.4171/JEMS/905},
       URL = {https://doi.org/10.4171/JEMS/905},
}

@Article{TT-IMRN,
 Author = {Tholozan, Nicolas and Tsouvalas, Konstantinos},
 Title = {Linearity and indiscreteness of amalgamated products of hyperbolic groups},
 FJournal = {IMRN. International Mathematics Research Notices},
 Journal = {Int. Math. Res. Not.},
 ISSN = {1073-7928},
 Volume = {2023},
 Number = {24},
 Pages = {21290--21319},
 Year = {2023},
 Language = {English},
 DOI = {10.1093/imrn/rnac302},
 zbMATH = {7805883}
}

@book{Mar,
  title={Discrete subgroups of semisimple Lie groups},
  author={Margulis, Gregori A.},
  volume={17},
  year={1991},
  publisher={Springer Science \& Business Media}
}

@Article{Malcev,
Author = {Mal'cev, A. I.},
Title = {On homomorphisms onto finite groups},
Journal = {Uchen. Zap. Ivanov Gos. Ped. Inst.},
Volume = {18},
Pages = {49--60},
Year = {1956},
Language = {Russian}
}

@ARTICLE{2024arXiv240317964K,
       author = {{Kharlampovich}, Olga and {Vdovina}, Alina},
        title = "{Quantifying separability in RAAGs via representations}",
      journal = {arXiv e-prints},
         year = 2024,
        month = mar,
          eid = {arXiv:2403.17964},
        pages = {arXiv:2403.17964},
          doi = {10.48550/arXiv.2403.17964},
archivePrefix = {arXiv},
       eprint = {2403.17964},
 primaryClass = {math.GR},
       adsurl = {https://ui.adsabs.harvard.edu/abs/2024arXiv240317964K}
}

@article {MR2100678,
    AUTHOR = {McReynolds, D. B.},
     TITLE = {Peripheral separability and cusps of arithmetic hyperbolic
              orbifolds},
   JOURNAL = {Algebr. Geom. Topol.},
  FJOURNAL = {Algebraic \& Geometric Topology},
    VOLUME = {4},
      YEAR = {2004},
     PAGES = {721--755},
      ISSN = {1472-2747,1472-2739},
   MRCLASS = {57M50 (11F06 20E26 30F40 57N16 57R90)},
  MRNUMBER = {2100678},
MRREVIEWER = {Kerry\ N.\ Jones},
       DOI = {10.2140/agt.2004.4.721},
       URL = {https://doi.org/10.2140/agt.2004.4.721},
}

@article {MR2511587,
    AUTHOR = {Prasad, Gopal and Rapinchuk, Andrei S.},
     TITLE = {Weakly commensurable arithmetic groups and isospectral locally
              symmetric spaces},
   JOURNAL = {Publ. Math. Inst. Hautes \'Etudes Sci.},
  FJOURNAL = {Publications Math\'ematiques. Institut de Hautes \'Etudes
              Scientifiques},
    NUMBER = {109},
      YEAR = {2009},
     PAGES = {113--184},
      ISSN = {0073-8301,1618-1913},
   MRCLASS = {20G15 (11E72 14L35 20G30)},
  MRNUMBER = {2511587},
MRREVIEWER = {B.\ Sury},
       DOI = {10.1007/s10240-009-0019-6},
       URL = {https://doi.org/10.1007/s10240-009-0019-6},
}

@article {MR4474914,
    AUTHOR = {Emery, Vincent and Kim, Inkang},
     TITLE = {Quaternionic hyperbolic lattices of minimal covolume},
   JOURNAL = {Forum Math. Sigma},
  FJOURNAL = {Forum of Mathematics. Sigma},
    VOLUME = {10},
      YEAR = {2022},
     PAGES = {Paper No. e68, 19},
      ISSN = {2050-5094},
   MRCLASS = {22E40 (11E57 20G30 51M25)},
  MRNUMBER = {4474914},
MRREVIEWER = {Jack\ O.\ Button},
       DOI = {10.1017/fms.2022.43},
       URL = {https://doi.org/10.1017/fms.2022.43},
}

@article {MR44570,
    AUTHOR = {Chevalley, Claude},
     TITLE = {Deux th\'eor\`emes d'arithm\'etique},
   JOURNAL = {J. Math. Soc. Japan},
  FJOURNAL = {Journal of the Mathematical Society of Japan},
    VOLUME = {3},
      YEAR = {1951},
     PAGES = {36--44},
      ISSN = {0025-5645,1881-1167},
   MRCLASS = {10.0X},
  MRNUMBER = {44570},
MRREVIEWER = {G.\ Hochschild},
       DOI = {10.2969/jmsj/00310036},
       URL = {https://doi.org/10.2969/jmsj/00310036},
}

@article {MR302822,
    AUTHOR = {Prasad, Gopal and Raghunathan, M. S.},
     TITLE = {Cartan subgroups and lattices in semi-simple groups},
   JOURNAL = {Ann. of Math. (2)},
  FJOURNAL = {Annals of Mathematics. Second Series},
    VOLUME = {96},
      YEAR = {1972},
     PAGES = {296--317},
      ISSN = {0003-486X},
   MRCLASS = {22E40 (53C30)},
  MRNUMBER = {302822},
MRREVIEWER = {J.\ A.\ Wolf},
       DOI = {10.2307/1970790},
       URL = {https://doi.org/10.2307/1970790},
}

@article {MR1960120,
    AUTHOR = {Prasad, Gopal and Rapinchuk, Andrei S.},
     TITLE = {Existence of irreducible {$\Bbb R$}-regular elements in
              {Z}ariski-dense subgroups},
   JOURNAL = {Math. Res. Lett.},
  FJOURNAL = {Mathematical Research Letters},
    VOLUME = {10},
      YEAR = {2003},
    NUMBER = {1},
     PAGES = {21--32},
      ISSN = {1073-2780},
   MRCLASS = {20G20},
  MRNUMBER = {1960120},
MRREVIEWER = {E.\ A.\ Tevel\"ev},
       DOI = {10.4310/MRL.2003.v10.n1.a3},
       URL = {https://doi.org/10.4310/MRL.2003.v10.n1.a3},
}

@article {MR3663601,
    AUTHOR = {Louder, Larsen and McReynolds, D. B. and Patel, Priyam},
     TITLE = {Zariski closures and subgroup separability},
   JOURNAL = {Selecta Math. (N.S.)},
  FJOURNAL = {Selecta Mathematica. New Series},
    VOLUME = {23},
      YEAR = {2017},
    NUMBER = {3},
     PAGES = {2019--2027},
      ISSN = {1022-1824,1420-9020},
   MRCLASS = {20E05 (20E26)},
  MRNUMBER = {3663601},
MRREVIEWER = {David\ I.\ Moldavanski\u i},
       DOI = {10.1007/s00029-016-0300-8},
       URL = {https://doi.org/10.1007/s00029-016-0300-8},
}

@article{zbMATH03152734,
 author = {Berger, Marcel},
 title = {Les espaces sym{\'e}triques non compacts},
 fjournal = {Annales Scientifiques de l'{\'E}cole Normale Sup{\'e}rieure. Troisi{\`e}me S{\'e}rie},
 journal = {Ann. Sci. {\'E}c. Norm. Sup{\'e}r. (3)},
 issn = {0012-9593},
 volume = {74},
 pages = {85--177},
 year = {1957},
 language = {French},
 doi = {10.24033/asens.1054},
 url = {https://eudml.org/doc/81724},
 zbMATH = {3152734},
 Zbl = {0093.35602}
}

@misc{arXiv:2105.06897,
 author = {Belolipetsky, Mikhail and Bogachev, Nikolay and Kolpakov, Alexander and Slavich, Leone},
 title = {Subspace stabilisers in hyperbolic lattices},
 year = {2025},
 howpublished = {Preprint, {arXiv}:2105.06897 [math.{GT}] (2025)},
 url = {https://arxiv.org/abs/2105.06897},
 arXiv = {arXiv:2105.06897}
}

@incollection {MR2655311,
    AUTHOR = {Benoist, Yves},
     TITLE = {Five lectures on lattices in semisimple {L}ie groups},
 BOOKTITLE = {G\'eom\'etries \`a{} courbure n\'egative ou nulle, groupes
              discrets et rigidit\'es},
    SERIES = {S\'emin. Congr.},
    VOLUME = {18},
     PAGES = {117--176},
 PUBLISHER = {Soc. Math. France, Paris},
      YEAR = {2009},
      ISBN = {978-2-85629-240-2},
   MRCLASS = {22E40 (11F06 20H10)},
  MRNUMBER = {2655311},
MRREVIEWER = {Dave\ Witte\ Morris},
}

@article {MR3708964,
    AUTHOR = {Parkkonen, Jouni and Paulin, Fr\'ed\'eric},
     TITLE = {A classification of {$\Bbb C$}-{F}uchsian subgroups of
              {P}icard modular groups},
   JOURNAL = {Math. Scand.},
  FJOURNAL = {Mathematica Scandinavica},
    VOLUME = {121},
      YEAR = {2017},
    NUMBER = {1},
     PAGES = {57--74},
      ISSN = {0025-5521,1903-1807},
   MRCLASS = {11F06 (11R52)},
  MRNUMBER = {3708964},
MRREVIEWER = {Matthew\ Stover},
       DOI = {10.7146/math.scand.a-26128},
       URL = {https://doi.org/10.7146/math.scand.a-26128},
}

@article {MR3705536,
    AUTHOR = {Chinburg, Ted and Stover, Matthew},
     TITLE = {Geodesic curves on {S}himura surfaces},
   JOURNAL = {Topology Proc.},
  FJOURNAL = {Topology Proceedings},
    VOLUME = {52},
      YEAR = {2018},
     PAGES = {113--121},
      ISSN = {0146-4124,2331-1290},
   MRCLASS = {20H10 (11F06 14G35)},
  MRNUMBER = {3705536},
MRREVIEWER = {Christian\ Frederik\ Wei\ss},
}

@misc{arXiv:math/0612290,
 author = {Agol, Ian},
 title = {Systoles of hyperbolic 4-manifolds},
 year = {2006},
 howpublished = {Preprint, {arXiv}:math/0612290 [math.{GT}] (2006)},
 url = {https://arxiv.org/abs/math/0612290},
 arXiv = {arXiv:math/0612290}
}

@article {MR2979855,
    AUTHOR = {Haglund, Fr\'ed\'eric and Wise, Daniel T.},
     TITLE = {A combination theorem for special cube complexes},
   JOURNAL = {Ann. of Math. (2)},
  FJOURNAL = {Annals of Mathematics. Second Series},
    VOLUME = {176},
      YEAR = {2012},
    NUMBER = {3},
     PAGES = {1427--1482},
      ISSN = {0003-486X,1939-8980},
   MRCLASS = {20F67 (20E06 20E26)},
  MRNUMBER = {2979855},
MRREVIEWER = {Yago\ Antol\'in},
       DOI = {10.4007/annals.2012.176.3.2},
       URL = {https://doi.org/10.4007/annals.2012.176.3.2},
}

@Article{Wehrfritz,
 Author = {Wehrfritz, B. A. F.},
 Title = {Generalized free products of linear groups},
 FJournal = {Proceedings of the London Mathematical Society. Third Series},
 Journal = {Proc. Lond. Math. Soc. (3)},
 ISSN = {0024-6115},
 Volume = {27},
 Pages = {402--424},
 Year = {1973},
 Language = {English},
 DOI = {10.1112/plms/s3-27.3.402},
 zbMATH = {3419405},
 Zbl = {0266.20052}
}

@article {MR2417445,
    AUTHOR = {Baker, Mark and Cooper, Daryl},
     TITLE = {A combination theorem for convex hyperbolic manifolds, with
              applications to surfaces in 3-manifolds},
   JOURNAL = {J. Topol.},
  FJOURNAL = {Journal of Topology},
    VOLUME = {1},
      YEAR = {2008},
    NUMBER = {3},
     PAGES = {603--642},
      ISSN = {1753-8416,1753-8424},
   MRCLASS = {57M50 (30F40)},
  MRNUMBER = {2417445},
MRREVIEWER = {David\ A.\ Dumas},
       DOI = {10.1112/jtopol/jtn013},
       URL = {https://doi.org/10.1112/jtopol/jtn013},
}

@book {MR240238,
    AUTHOR = {Bourbaki, N.},
     TITLE = {\'El\'ements de math\'ematique. {F}asc. {XXXIV}. {G}roupes et
              alg\`ebres de {L}ie. {C}hapitre {IV}: {G}roupes de {C}oxeter
              et syst\`emes de {T}its. {C}hapitre {V}: {G}roupes engendr\'es
              par des r\'eflexions. {C}hapitre {VI}: syst\`emes de racines},
    SERIES = {Actualit\'es Scientifiques et Industrielles [Current
              Scientific and Industrial Topics]},
    VOLUME = {No. 1337},
 PUBLISHER = {Hermann, Paris},
      YEAR = {1968},
     PAGES = {288 pp. (loose errata)},
   MRCLASS = {22.50 (17.00)},
  MRNUMBER = {240238},
MRREVIEWER = {G.\ B.\ Seligman},
}

@article {MR302779,
    AUTHOR = {Vinberg, \`E.\ B.},
     TITLE = {Discrete linear groups that are generated by reflections},
   JOURNAL = {Izv. Akad. Nauk SSSR Ser. Mat.},
  FJOURNAL = {Izvestiya Akademii Nauk SSSR. Seriya Matematicheskaya},
    VOLUME = {35},
      YEAR = {1971},
     PAGES = {1072--1112},
      ISSN = {0373-2436},
   MRCLASS = {20H15 (50B30)},
  MRNUMBER = {302779},
MRREVIEWER = {J.\ C.\ Fisher},
}

@article {MR2130566,
    AUTHOR = {Potyagailo, Leonid and Vinberg, Ernest},
     TITLE = {On right-angled reflection groups in hyperbolic spaces},
   JOURNAL = {Comment. Math. Helv.},
  FJOURNAL = {Commentarii Mathematici Helvetici. A Journal of the Swiss
              Mathematical Society},
    VOLUME = {80},
      YEAR = {2005},
    NUMBER = {1},
     PAGES = {63--73},
      ISSN = {0010-2571,1420-8946},
   MRCLASS = {20F55 (51F15 57M07 57M50)},
  MRNUMBER = {2130566},
MRREVIEWER = {Ruth\ Kellerhals},
       DOI = {10.4171/CMH/4},
       URL = {https://doi.org/10.4171/CMH/4},
}

@misc{arXiv:2605.21734,
 author = {Changqian Li},
 title = {Virtual specialness of the double},
 year = {2026},
 howpublished = {Preprint, {arXiv}:2605.21734 [math.{GR}] (2026)},
 url = {https://arxiv.org/abs/2605.21734},
 arXiv = {arXiv:2605.21734}
}

@article {MR4693936,
    AUTHOR = {Huang, Jingyin and Wise, Daniel T.},
     TITLE = {Virtual specialness of certain graphs of special cube
              complexes},
   JOURNAL = {Math. Ann.},
  FJOURNAL = {Mathematische Annalen},
    VOLUME = {388},
      YEAR = {2024},
    NUMBER = {1},
     PAGES = {329--357},
      ISSN = {0025-5831,1432-1807},
   MRCLASS = {55U10 (05E45 20F65)},
  MRNUMBER = {4693936},
MRREVIEWER = {Sam\ Shepherd},
       DOI = {10.1007/s00208-022-02527-0},
       URL = {https://doi.org/10.1007/s00208-022-02527-0},
}

@misc{danciger2024combination,
 author = {Jeffrey Danciger and Fran{\c{c}}ois Gu{\'e}ritaud and Fanny Kassel},
 title = {Combination theorems in convex projective geometry},
 year = {2025},
 howpublished = {Preprint, {arXiv}:2407.09439 [math.{GR}] (2025)},
 url = {https://arxiv.org/abs/2407.09439},
 arXiv = {arXiv:2407.09439}
}

@article {MR2587434,
    AUTHOR = {Manning, Jason Fox and Mart\'inez-Pedroza, Eduardo},
     TITLE = {Separation of relatively quasiconvex subgroups},
   JOURNAL = {Pacific J. Math.},
  FJOURNAL = {Pacific Journal of Mathematics},
    VOLUME = {244},
      YEAR = {2010},
    NUMBER = {2},
     PAGES = {309--334},
      ISSN = {0030-8730,1945-5844},
   MRCLASS = {20F67 (20E26 57M50)},
  MRNUMBER = {2587434},
MRREVIEWER = {Vassilis\ Metaftsis},
       DOI = {10.2140/pjm.2010.244.309},
       URL = {https://doi.org/10.2140/pjm.2010.244.309},
}

@book {MR4298722,
    AUTHOR = {Wise, Daniel T.},
     TITLE = {The structure of groups with a quasiconvex hierarchy},
    SERIES = {Annals of Mathematics Studies},
    VOLUME = {209},
 PUBLISHER = {Princeton University Press, Princeton, NJ},
      YEAR = {[2021] \copyright 2021},
     PAGES = {x+357},
      ISBN = {[9780691170442]; [9780691170459]; [9780691213507]},
   MRCLASS = {20F65 (20F67)},
  MRNUMBER = {4298722},
MRREVIEWER = {Anthony\ Genevois},
}

@article {HW08,
    AUTHOR = {Haglund, Fr\'ed\'eric and Wise, Daniel T.},
     TITLE = {Special cube complexes},
   JOURNAL = {Geom. Funct. Anal.},
  FJOURNAL = {Geometric and Functional Analysis},
    VOLUME = {17},
      YEAR = {2008},
    NUMBER = {5},
     PAGES = {1551--1620},
      ISSN = {1016-443X,1420-8970},
   MRCLASS = {20F36 (20F55 20F67)},
  MRNUMBER = {2377497},
MRREVIEWER = {Patrick\ Bahls},
       DOI = {10.1007/s00039-007-0629-4},
       URL = {https://doi.org/10.1007/s00039-007-0629-4},
}

@Article{GS,
 Author = {Gromov, Mikhail and Schoen, Richard},
 Title = {Harmonic maps into singular spaces and {{\(p\)}}-adic superrigidity for lattices in groups of rank one},
 FJournal = {Publications Math{\'e}matiques},
 Journal = {Publ. Math., Inst. Hautes {\'E}tud. Sci.},
 ISSN = {0073-8301},
 Volume = {76},
 Pages = {165--246},
 Year = {1992},
 Language = {English},
 DOI = {10.1007/BF02699433},
 zbMATH = {205186},
 Zbl = {0896.58024}
}

@article {MR2821431,
    AUTHOR = {Belolipetsky, Mikhail V. and Thomson, Scott A.},
     TITLE = {Systoles of hyperbolic manifolds},
   JOURNAL = {Algebr. Geom. Topol.},
  FJOURNAL = {Algebraic \& Geometric Topology},
    VOLUME = {11},
      YEAR = {2011},
    NUMBER = {3},
     PAGES = {1455--1469},
      ISSN = {1472-2747,1472-2739},
   MRCLASS = {53C23 (57M50)},
  MRNUMBER = {2821431},
MRREVIEWER = {Mikhail\ G.\ Katz},
       DOI = {10.2140/agt.2011.11.1455},
       URL = {https://doi.org/10.2140/agt.2011.11.1455},
}

@book{On-Vin,
  title={Lie groups and algebraic groups},
  author={Onishchik, Arkadij and Vinberg, Ernest},
  year={2012},
  publisher={Springer Science \& Business Media}
}

@article {MR4307692,
    AUTHOR = {Minasyan, Ashot},
     TITLE = {Virtual retraction properties in groups},
   JOURNAL = {Int. Math. Res. Not. IMRN},
  FJOURNAL = {International Mathematics Research Notices. IMRN},
      YEAR = {2021},
    NUMBER = {17},
     PAGES = {13434--13477},
      ISSN = {1073-7928,1687-0247},
   MRCLASS = {20E26 (20E06 20E07 20E34)},
  MRNUMBER = {4307692},
MRREVIEWER = {Motiejus\ Valiunas},
       DOI = {10.1093/imrn/rnz249},
       URL = {https://doi.org/10.1093/imrn/rnz249},
}

@article {MR3001608,
    AUTHOR = {Namazi, Hossein and Souto, Juan},
     TITLE = {Non-realizability and ending laminations: proof of the density
              conjecture},
   JOURNAL = {Acta Math.},
  FJOURNAL = {Acta Mathematica},
    VOLUME = {209},
      YEAR = {2012},
    NUMBER = {2},
     PAGES = {323--395},
      ISSN = {0001-5962,1871-2509},
   MRCLASS = {30F40 (20H10 57M50)},
  MRNUMBER = {3001608},
MRREVIEWER = {Majid\ Heydarpour},
       DOI = {10.1007/s11511-012-0088-0},
       URL = {https://doi.org/10.1007/s11511-012-0088-0},
}

@article {MR2925381,
    AUTHOR = {Brock, Jeffrey F. and Canary, Richard D. and Minsky, Yair N.},
     TITLE = {The classification of {K}leinian surface groups, {II}: {T}he
              ending lamination conjecture},
   JOURNAL = {Ann. of Math. (2)},
  FJOURNAL = {Annals of Mathematics. Second Series},
    VOLUME = {176},
      YEAR = {2012},
    NUMBER = {1},
     PAGES = {1--149},
      ISSN = {0003-486X,1939-8980},
   MRCLASS = {57M50 (30F40)},
  MRNUMBER = {2925381},
MRREVIEWER = {Athanase\ Papadopoulos},
       DOI = {10.4007/annals.2012.176.1.1},
       URL = {https://doi.org/10.4007/annals.2012.176.1.1},
}

@article {MR2821565,
    AUTHOR = {Ohshika, Ken'ichi},
     TITLE = {Realising end invariants by limits of minimally parabolic,
              geometrically finite groups},
   JOURNAL = {Geom. Topol.},
  FJOURNAL = {Geometry \& Topology},
    VOLUME = {15},
      YEAR = {2011},
    NUMBER = {2},
     PAGES = {827--890},
      ISSN = {1465-3060,1364-0380},
   MRCLASS = {57M50 (30F40)},
  MRNUMBER = {2821565},
MRREVIEWER = {Bruno\ P.\ Zimmermann},
       DOI = {10.2140/gt.2011.15.827},
       URL = {https://doi.org/10.2140/gt.2011.15.827},
}

@article {MR860677,
    AUTHOR = {Brooks, Robert},
     TITLE = {Circle packings and co-compact extensions of {K}leinian
              groups},
   JOURNAL = {Invent. Math.},
  FJOURNAL = {Inventiones Mathematicae},
    VOLUME = {86},
      YEAR = {1986},
    NUMBER = {3},
     PAGES = {461--469},
      ISSN = {0020-9910,1432-1297},
   MRCLASS = {32G15 (30F40 52A45)},
  MRNUMBER = {860677},
MRREVIEWER = {Tadashi\ Kuroda},
       DOI = {10.1007/BF01389263},
       URL = {https://doi.org/10.1007/BF01389263},
}

@article {MR2079598,
    AUTHOR = {Brock, Jeffrey F. and Bromberg, Kenneth W.},
     TITLE = {On the density of geometrically finite {K}leinian groups},
   JOURNAL = {Acta Math.},
  FJOURNAL = {Acta Mathematica},
    VOLUME = {192},
      YEAR = {2004},
    NUMBER = {1},
     PAGES = {33--93},
      ISSN = {0001-5962,1871-2509},
   MRCLASS = {57M50 (30F40)},
  MRNUMBER = {2079598},
MRREVIEWER = {Lee\ Mosher},
       DOI = {10.1007/BF02441085},
       URL = {https://doi.org/10.1007/BF02441085},
}

@article {MR2980525,
    AUTHOR = {Chesebro, Eric and DeBlois, Jason and Wilton, Henry},
     TITLE = {Some virtually special hyperbolic 3-manifold groups},
   JOURNAL = {Comment. Math. Helv.},
  FJOURNAL = {Commentarii Mathematici Helvetici. A Journal of the Swiss
              Mathematical Society},
    VOLUME = {87},
      YEAR = {2012},
    NUMBER = {3},
     PAGES = {727--787},
      ISSN = {0010-2571,1420-8946},
   MRCLASS = {57M10 (20E26 20F55)},
  MRNUMBER = {2980525},
MRREVIEWER = {Riikka\ Kangaslampi},
       DOI = {10.4171/CMH/267},
       URL = {https://doi.org/10.4171/CMH/267},
}

@article {MR2684983,
    AUTHOR = {Hruska, G. Christopher},
     TITLE = {Relative hyperbolicity and relative quasiconvexity for
              countable groups},
   JOURNAL = {Algebr. Geom. Topol.},
  FJOURNAL = {Algebraic \& Geometric Topology},
    VOLUME = {10},
      YEAR = {2010},
    NUMBER = {3},
     PAGES = {1807--1856},
      ISSN = {1472-2747,1472-2739},
   MRCLASS = {20F65 (20F67)},
  MRNUMBER = {2684983},
MRREVIEWER = {Eduardo\ Mart\'inez-Pedroza},
       DOI = {10.2140/agt.2010.10.1807},
       URL = {https://doi.org/10.2140/agt.2010.10.1807},
}

@book {MR3307755,
    AUTHOR = {Morris, Dave Witte},
     TITLE = {Introduction to arithmetic groups},
 PUBLISHER = {Deductive Press, [place of publication not identified]},
      YEAR = {2015},
     PAGES = {xii+475},
      ISBN = {978-0-9865716-0-2; 978-0-9865716-1-9},
   MRCLASS = {22E40 (11F70 20G20 20G30 22D10 22D40)},
  MRNUMBER = {3307755},
MRREVIEWER = {S.\ G.\ Dani},
}

@article{Nisnevich,
  title={On groups isomorphically representable by matrices over a commutative field},
  author={Nisnevich, VL},
  journal={Mat. Sb},
  volume={8},
  pages={395--403},
  year={1940}
}

@article {MR3104553,
    AUTHOR = {Agol, Ian},
     TITLE = {The virtual {H}aken conjecture},
      NOTE = {With an appendix by Agol, Daniel Groves, and Jason Manning},
   JOURNAL = {Doc. Math.},
  FJOURNAL = {Documenta Mathematica},
    VOLUME = {18},
      YEAR = {2013},
     PAGES = {1045--1087},
      ISSN = {1431-0635,1431-0643},
   MRCLASS = {20F67 (57Mxx)},
  MRNUMBER = {3104553},
MRREVIEWER = {Thomas\ Koberda},
       URL = {https://elibm.org/article/10000267},
}

@article {MR587545,
    AUTHOR = {Leung, D. S. P.},
     TITLE = {Reflective submanifolds. {III}. {C}ongruency of isometric
              reflective submanifolds and corrigenda to the classification
              of reflective submanifolds},
   JOURNAL = {J. Differential Geometry},
  FJOURNAL = {Journal of Differential Geometry},
    VOLUME = {14},
      YEAR = {1979},
    NUMBER = {2},
     PAGES = {167--177},
      ISSN = {0022-040X,1945-743X},
   MRCLASS = {53C35 (53C40)},
  MRNUMBER = {587545},
MRREVIEWER = {J.\ Szenthe},
       URL = {http://projecteuclid.org/euclid.jdg/1214434966},
}

@article {MR367873,
    AUTHOR = {Leung, Dominic S. P.},
     TITLE = {On the classification of reflective submanifolds of
              {R}iemannian symmetric spaces},
   JOURNAL = {Indiana Univ. Math. J.},
  FJOURNAL = {Indiana University Mathematics Journal},
    VOLUME = {24},
      YEAR = {1974/75},
     PAGES = {327--339},
      ISSN = {0022-2518,1943-5258},
   MRCLASS = {53C35},
  MRNUMBER = {367873},
MRREVIEWER = {J.\ Szenthe},
       DOI = {10.1512/iumj.1974.24.24029},
       URL = {https://doi.org/10.1512/iumj.1974.24.24029},
}

@MISC {agol2018hyperbolic,
    TITLE = "{Hyperbolic $3$-manifold groups that embed in compact Lie groups}",
    AUTHOR = {Agol, Ian},
    year={2018},
    HOWPUBLISHED = {MathOverflow},
    NOTE = {URL:https://mathoverflow.net/q/315430 (version: 2018-11-16)},
    EPRINT = {https://mathoverflow.net/q/315430},
    URL = {https://mathoverflow.net/q/315430}
}

@article {MR549966,
    AUTHOR = {Kneser, Martin},
     TITLE = {Normalteiler ganzzahliger {S}pingruppen},
   JOURNAL = {J. Reine Angew. Math.},
  FJOURNAL = {Journal f\"ur die Reine und Angewandte Mathematik. [Crelle's
              Journal]},
    VOLUME = {311/312},
      YEAR = {1979},
     PAGES = {191--214},
      ISSN = {0075-4102,1435-5345},
   MRCLASS = {20G30 (10C30 20H05 22E40)},
  MRNUMBER = {549966},
MRREVIEWER = {Carl\ Riehm},
       DOI = {10.1515/crll.1979.311-312.191},
       URL = {https://doi.org/10.1515/crll.1979.311-312.191},
}

@article{zbMATH03182231,
 author = {Baumslag, G.},
 title = {On the residual finiteness of generalized free products of nilpotent groups},
 fjournal = {Transactions of the American Mathematical Society},
 journal = {Trans. Am. Math. Soc.},
 issn = {0002-9947},
 volume = {106},
 pages = {193--209},
 year = {1963},
 language = {English},
 doi = {10.2307/1993762},
 zbMATH = {3182231},
 Zbl = {0112.25904}
}

@article {MR4862320,
    AUTHOR = {Tholozan, Nicolas and Tsouvalas, Konstantinos},
     TITLE = {Residually finite non-linear hyperbolic groups},
   JOURNAL = {Comment. Math. Helv.},
  FJOURNAL = {Commentarii Mathematici Helvetici. A Journal of the Swiss
              Mathematical Society},
    VOLUME = {100},
      YEAR = {2025},
    NUMBER = {1},
     PAGES = {1--9},
      ISSN = {0010-2571,1420-8946},
   MRCLASS = {20F67 (20E26 22E40)},
  MRNUMBER = {4862320},
MRREVIEWER = {Monika\ Kudlinska},
       DOI = {10.4171/cmh/581},
       URL = {https://doi.org/10.4171/cmh/581},
}

@article {MR4653708,
    AUTHOR = {Chong, Hip Kuen and Wise, Daniel T.},
     TITLE = {Continuously many quasi-isometry classes of residually finite
              groups},
   JOURNAL = {Glasg. Math. J.},
  FJOURNAL = {Glasgow Mathematical Journal},
    VOLUME = {65},
      YEAR = {2023},
    NUMBER = {3},
     PAGES = {569--572},
      ISSN = {0017-0895,1469-509X},
   MRCLASS = {20E26 (20F06)},
  MRNUMBER = {4653708},
MRREVIEWER = {Jack\ O.\ Button},
       DOI = {10.1017/s0017089523000137},
       URL = {https://doi.org/10.1017/s0017089523000137},
}

@article {MR422501,
    AUTHOR = {Millson, John J.},
     TITLE = {On the first {B}etti number of a constant negatively curved
              manifold},
   JOURNAL = {Ann. of Math. (2)},
  FJOURNAL = {Annals of Mathematics. Second Series},
    VOLUME = {104},
      YEAR = {1976},
    NUMBER = {2},
     PAGES = {235--247},
      ISSN = {0003-486X},
   MRCLASS = {22E40},
  MRNUMBER = {422501},
MRREVIEWER = {\`E.\ Vinberg},
       DOI = {10.2307/1971046},
       URL = {https://doi.org/10.2307/1971046},
}

@article {MR4388367,
    AUTHOR = {Chong, Hip Kuen and Wise, Daniel T.},
     TITLE = {An uncountable family of finitely generated residually finite
              groups},
   JOURNAL = {J. Group Theory},
  FJOURNAL = {Journal of Group Theory},
    VOLUME = {25},
      YEAR = {2022},
    NUMBER = {2},
     PAGES = {207--216},
      ISSN = {1433-5883,1435-4446},
   MRCLASS = {20E26 (20E22 20F05)},
  MRNUMBER = {4388367},
MRREVIEWER = {Valeriy\ G.\ Bardakov},
       DOI = {10.1515/jgth-2021-0094},
       URL = {https://doi.org/10.1515/jgth-2021-0094},
}

@article {MR764305,
    AUTHOR = {Grigorchuk, R. I.},
     TITLE = {Degrees of growth of finitely generated groups and the theory
              of invariant means},
   JOURNAL = {Izv. Akad. Nauk SSSR Ser. Mat.},
  FJOURNAL = {Izvestiya Akademii Nauk SSSR. Seriya Matematicheskaya},
    VOLUME = {48},
      YEAR = {1984},
    NUMBER = {5},
     PAGES = {939--985},
      ISSN = {0373-2436},
   MRCLASS = {20F05 (43A07)},
  MRNUMBER = {764305},
MRREVIEWER = {P.\ Gerl},
}

@article {MR2776645,
    AUTHOR = {Bergeron, Nicolas and Haglund, Fr\'ed\'eric and Wise, Daniel
              T.},
     TITLE = {Hyperplane sections in arithmetic hyperbolic manifolds},
   JOURNAL = {J. Lond. Math. Soc. (2)},
  FJOURNAL = {Journal of the London Mathematical Society. Second Series},
    VOLUME = {83},
      YEAR = {2011},
    NUMBER = {2},
     PAGES = {431--448},
      ISSN = {0024-6107,1469-7750},
   MRCLASS = {57M50 (11F75 53C23)},
  MRNUMBER = {2776645},
MRREVIEWER = {Pablo\ Su\'arez-Serrato},
       DOI = {10.1112/jlms/jdq082},
       URL = {https://doi.org/10.1112/jlms/jdq082},
}

@book {MR559927,
    AUTHOR = {Jacobson, Nathan},
     TITLE = {Lie algebras},
      NOTE = {Republication of the 1962 original},
 PUBLISHER = {Dover Publications, Inc., New York},
      YEAR = {1979},
     PAGES = {ix+331},
      ISBN = {0-486-63832-4},
   MRCLASS = {17-01},
  MRNUMBER = {559927},
}

@article {MR1769939,
    AUTHOR = {Bergeron, Nicolas},
     TITLE = {Premier nombre de {B}etti et spectre du laplacien de certaines
              vari\'et\'es hyperboliques},
   JOURNAL = {Enseign. Math. (2)},
  FJOURNAL = {L'Enseignement Math\'ematique. Revue Internationale. 2e
              S\'erie},
    VOLUME = {46},
      YEAR = {2000},
    NUMBER = {1-2},
     PAGES = {109--137},
      ISSN = {0013-8584},
   MRCLASS = {58J53 (53C20 53C22 53C23 57M50)},
  MRNUMBER = {1769939},
MRREVIEWER = {Athanase\ Papadopoulos},
}

@misc{arXiv:math/0405568,
 author = {Agol, Ian},
 title = {Tameness of hyperbolic 3-manifolds},
 year = {2004},
 howpublished = {Preprint, {arXiv}:math/0405568 [math.{GT}] (2004)},
 url = {https://arxiv.org/abs/math/0405568},
 arXiv = {arXiv:math/0405568}
}

@article {MR2188131,
    AUTHOR = {Calegari, Danny and Gabai, David},
     TITLE = {Shrinkwrapping and the taming of hyperbolic 3-manifolds},
   JOURNAL = {J. Amer. Math. Soc.},
  FJOURNAL = {Journal of the American Mathematical Society},
    VOLUME = {19},
      YEAR = {2006},
    NUMBER = {2},
     PAGES = {385--446},
      ISSN = {0894-0347,1088-6834},
   MRCLASS = {57M50 (30F40 57N10)},
  MRNUMBER = {2188131},
MRREVIEWER = {Bruno\ P.\ Zimmermann},
       DOI = {10.1090/S0894-0347-05-00513-8},
       URL = {https://doi.org/10.1090/S0894-0347-05-00513-8},
}

@article {MR1396777,
    AUTHOR = {Canary, Richard D.},
     TITLE = {A covering theorem for hyperbolic {$3$}-manifolds and its
              applications},
   JOURNAL = {Topology},
  FJOURNAL = {Topology. An International Journal of Mathematics},
    VOLUME = {35},
      YEAR = {1996},
    NUMBER = {3},
     PAGES = {751--778},
      ISSN = {0040-9383},
   MRCLASS = {57M50 (30F40 57N10)},
  MRNUMBER = {1396777},
MRREVIEWER = {Ken-ichi\ Ohshika},
       DOI = {10.1016/0040-9383(94)00055-7},
       URL = {https://doi.org/10.1016/0040-9383(94)00055-7},
}

@article {MR898729,
    AUTHOR = {Long, D. D.},
     TITLE = {Immersions and embeddings of totally geodesic surfaces},
   JOURNAL = {Bull. London Math. Soc.},
  FJOURNAL = {The Bulletin of the London Mathematical Society},
    VOLUME = {19},
      YEAR = {1987},
    NUMBER = {5},
     PAGES = {481--484},
      ISSN = {0024-6093,1469-2120},
   MRCLASS = {57N10},
  MRNUMBER = {898729},
MRREVIEWER = {G.\ Peter\ Scott},
       DOI = {10.1112/blms/19.5.481},
       URL = {https://doi.org/10.1112/blms/19.5.481},
}

@book{Humphreys,
  title={Linear algebraic groups},
  author={Humphreys, James E},
  volume={21},
  year={2012},
  publisher={Springer Science \& Business Media}
}

\end{document}